\documentclass[11pt,twoside]{article}
\usepackage[margin=1.3in]{geometry}

\usepackage{amsfonts,amsmath,amssymb,amsthm}
\usepackage{bm,color,xcolor,latexsym,mathrsfs}
\usepackage{cases,cite,enumitem,makeidx,times}
\usepackage[colorlinks,citecolor=violet,linkcolor=teal]{hyperref}
\usepackage[title,titletoc,toc]{appendix}

\let\oldbibliography\thebibliography
\renewcommand{\thebibliography}[1]{\oldbibliography{#1}\setlength{\itemsep}{0pt}}

\makeindex
\newtheorem{theorem}{Theorem}[section]
\newtheorem{corollary}[theorem]{Corollary}
\newtheorem{definition}[theorem]{Definition}
\newtheorem{lemma}[theorem]{Lemma}

\newtheorem{remark}[theorem]{Remark}

\numberwithin{equation}{section}

\newcommand{\N}{\mathbb N}
\newcommand{\Z}{\mathbb Z}
\newcommand{\Q}{\mathbb Q}
\newcommand{\R}{\mathbb R}

\newcommand{\Sp}{\mathbb S}

\begin{document}
\title{\textbf{Asymptotic expansions for conformal scalar curvature equations near isolated singularities}\bigskip}

\author{\medskip Xusheng Du \quad and \quad Hui Yang\footnote{H. Yang is partially supported by NSFC grant 12301140.}}

\date{\today}

\maketitle

\begin{center}
\begin{minipage}{130mm}

\begin{center}{\bf Abstract}\end{center}
In this paper, we study asymptotic expansions of positive solutions of the conformal scalar curvature equation
$$
- \Delta u = K(x) u^\frac{n + 2}{n - 2} ~~~~~~ \textmd{in} ~ B_1 \setminus \{ 0 \}
$$
with an isolated singularity at the origin. Under certain flatness conditions on $K$, we establish a higher-order expansion of solutions near the origin. In particular, we give the refined second-order asymptotic expansion of solutions when $n \geq 6$. Moreover, we also obtain an arbitrary-order expansion of singular positive solutions of the anisotropic elliptic equation
$$
- \,{\rm div} (|x|^{- 2 a} \nabla u) = |x|^{- b p} u^{p - 1} ~~~~~~ \textmd{in} ~ B_1 \setminus \{ 0 \},
$$
where $0 \leq a < \frac{n - 2}{2}$, $a \leq b < a + 1$ and $p = \frac{2 n}{n - 2 + 2 (b - a)}$. This equation is arising from the celebrated Caffarelli-Kohn-Nirenberg inequality.

\medskip

\noindent{\it Keywords}: Asymptotic expansions, scalar curvature equations, isolated singularities, critical exponents

\medskip

\noindent {\it MSC (2020)}: 35J61; 35C20

\end{minipage}
\end{center}

\tableofcontents

\section{Introduction}

In this paper, we are interested in the asymptotic behavior of solutions of the conformal scalar curvature equation
\begin{equation}\label{eq:csc}
- \Delta u = K(x) u^\frac{n + 2}{n - 2}, ~~~ u > 0 ~~~ \textmd{in} ~ B_1 \setminus \{ 0 \}
\end{equation}
with an isolated singularity at the origin, where $B_1$ is the open unit ball in $\R^n$ with $n \geq 3$ and $K \in C^1 (B_1)$ is a {\it positive} function. The equation \eqref{eq:csc} appears in the problem of finding a metric $g = u^{4/(n - 2)} \delta_{ij}$ conformal to the flat metric $\delta_{ij}$ on $\R^n$ such that $K(x)$ is the scalar curvature of the new metric. The classical works of Schoen and Yau \cite{Sch88,SY88} on complete conformally flat manifolds have indicated the importance of studying the singular distributional solutions of \eqref{eq:csc}. Solutions with an isolated singularity are the simplest examples of those singular solutions. The construction of singular solutions has been studied in \cite{Sch88,MP96,MP,CD,CL99,P} and the references therein.

The completeness of the conformal metric $g$ implies that $u$ is singular at the origin. Thus, a natural question is to understand how $u(x)$ tends to infinity when $x$ approaches the origin? When $K(x)$ is identically equal to $1$ (i.e., $K \equiv 1$), Caffarelli, Gidas and Spruck in their pioneering paper \cite{CGS} proved that every solution $u \in C^2 (B_1 \setminus \{ 0 \})$ of \eqref{eq:csc} with a non-removable singularity is asymptotically radially symmetric near $0$ and furthermore, there exists a global singular solution $u_0$ of
\begin{equation}\label{gloeq10}
\begin{cases}
- \Delta u_0 = u_0^\frac{n + 2}{n - 2}, ~~~ u_0 > 0 ~~~ \textmd{in} ~ \R^n \setminus \{ 0 \}, \\
u_0 \in C^2 (\R^n \setminus \{ 0 \}) ~~~ \textmd{and} ~~~ \lim_{|x| \to 0^+} u_0 (x) = + \infty
\end{cases}
\end{equation}
such that
\begin{equation}\label{asy=glo}
u(x) = u_0 (x) (1 + o(1)) ~~~~~~ \textmd{as} ~ |x| \to 0^+.
\end{equation}
They also showed that all solutions of \eqref{gloeq10} are radially symmetric, and hence a complete classification could be given via the standard phase-plane analysis. These global singular solutions are usually called the Fowler solutions or Delaunay type solutions. Subsequent to \cite{CGS}, second-order Yamabe type equations with isolated singularities have been extensively studied; see, for example, \cite{CHanY,HLL,HLT,HXZ,KMPS,LC96,Mar,XZ} and the references therein. In particular, Korevaar, Mazzeo, Pacard and Schoen in \cite{KMPS} studied refined asymptotics and expanded singular solutions $u$ to the first order. Recently, Han, Li and Li \cite{HLL} established the expansions up to arbitrary orders.

When the scalar curvature $K(x)$ is a positive and nonconstant $C^1$ function, Chen-Lin \cite{CL97,CL98,CL99a,L00} and Taliaferro-Zhang \cite{TZ06,Zhang02} investigated under what conditions on $K(x)$ every singular solution $u$ of \eqref{eq:csc} are still asymptotic to a radial singular solution $u_0$ of
\begin{equation}\label{eq:u0}
\begin{cases}
- \Delta u_0 = K(0) u_0^\frac{n + 2}{n - 2}, ~~~ u_0 > 0 ~~~ \textmd{in} ~ \R^n \setminus \{ 0 \}, \\
u_0 \in C^2 (\R^n \setminus \{ 0 \}) ~~~ \textmd{and} ~~~ \lim_{|x| \to 0^+} u_0 (x) = + \infty.
\end{cases}
\end{equation}
To further state the relevant results, we introduce the following notation and assumption.

\medskip

\noindent{\bf Notation $\mathcal{C}^\alpha (B_1).$}
For $\alpha > 0$, we introduce a function space $\mathcal{C}^\alpha (B_1)$.
\begin{enumerate}[label = \rm(\roman*)]
\item If $\alpha$ is a positive integer, then $\mathcal{C}^\alpha (B_1)$ is the usual space $C^\alpha (B_1)$.

\item If $\alpha > [\alpha]$, then $\mathcal{C}^\alpha (B_1)$ is the set of all functions $f \in C^{[\alpha]} (B_1)$ such that
$$
| \nabla^{[\alpha]} f(x) - \nabla^{[\alpha]} f(y) | \leq c(|x - y|) |x - y|^{\alpha - [\alpha]} ~~~~~~ \textmd{for all} ~ x, y \in B_1,
$$
where $c(\cdot)$ is a nonnegative continuous function satisfying $c(0) = 0$.
\end{enumerate}

\noindent{\bf Hypothesis H.}
$K$ is a positive function in $\mathcal{C}^\alpha (B_1)$ with $\alpha = (n - 2)/2$. Also, if $n \geq 6$, then
$$
| \nabla^j K(x) | \leq c(|x|) | \nabla K(x) |^\frac{\alpha - j}{\alpha - 1} ~~~~~~ \textmd{for all} ~ 2 \leq j \leq [\alpha], x \in B_1,
$$
where $c(\cdot)$ is a nonnegative continuous function satisfying $c(0) = 0$.

\medskip

Chen and Lin in a series of papers \cite{CL97,CL99a,L00} studied the asymptotic behavior of solutions of \eqref{eq:csc} under the assumption that $0 < c_1 |x|^{\alpha - 1} \leq | \nabla K(x) | \leq c_2 |x|^{\alpha - 1}$ near the origin for some $\alpha > 0$. Furthermore, Taliaferro and Zhang \cite{TZ06} showed that if $K \in C^1 (B_1)$ and Hypothesis H holds, then every singular solution $u$ of \eqref{eq:csc} satisfies \eqref{asy=glo} for some radial solution $u_0$ of \eqref{eq:u0}. In this paper, we will establish a higher-order expansion for singular solutions of \eqref{eq:csc}.

Suppose that $u_0$ is a Fowler solution of \eqref{eq:u0}. Let
$$
u_0 (|x|) = |x|^{ - \frac{n - 2}{2} } \xi (- \ln |x|),
$$
where $\xi (t)$ is a positive periodic solution (still called Fowler solution) of
\begin{equation}\label{eq:xi}
- \xi'' + \frac{(n - 2)^2}{4} \xi = K(0) \xi^\frac{n + 2}{n - 2} ~~~~~~ \textmd{in} ~ \R.
\end{equation}
Set $t = - \ln |x|$, $\theta = x/|x| \in \Sp^{n - 1}$ and
$$
v(t, \theta) = e^{ - \frac{n - 2}{2} t } u(e^{- t} \theta) ~~~~~~ \textmd{for} ~ (t, \theta) \in \R_+ \times \Sp^{n - 1}.
$$
Then $v$ solves the equation
\begin{equation}\label{eq:v}
- v_{tt} - \Delta_\theta v + \frac{(n - 2)^2}{4} v = K(t, \theta) v^\frac{n + 2}{n - 2} ~~~~~~ \textmd{in} ~ \R_+ \times \Sp^{n - 1},
\end{equation}
where $K(t, \theta) := K(e^{- t} \theta) = K(x)$ and $\Delta_\theta$ is the Laplace-Beltrami operator on $\Sp^{n - 1}$. Consider the linearized operator of \eqref{eq:v} at a Fowler solution $\xi$
\begin{equation}\label{eq:yp-L}
L = - \frac{\partial^2}{\partial t^2} - \Delta_\theta + \frac{(n - 2)^2}{4} - \frac{n + 2}{n - 2} K(0) \xi^\frac{4}{n - 2},
\end{equation}
which has periodic coefficients and hence may be studied by classical Floquet theoretic method as in \cite{MPU}. The linear operator $L$ can be decoupled into infinitely many ordinary differential operators using the eigendata of $\Delta_\theta$. More precisely, let $\{ \lambda_i \}_{i \geq 0}$ be the sequence of eigenvalues (in the increasing order with multiplicity) of $- \Delta_\theta$ on $\Sp^{n - 1}$, and $\{ X_i \}_{i \geq 0}$ be a sequence of the corresponding normalized eigenfunctions of $- \Delta_\theta$ on $L^2 (\Sp^{n - 1})$, i.e., for each $i \geq 0$,
$$
- \Delta_\theta X_i = \lambda_i X_i ~~~~~~ \textmd{on} ~ \Sp^{n - 1}.
$$
Recall that $\lambda_0 = 0$, $\lambda_1 = \cdots = \lambda_n = n - 1$, $\lambda_{n + 1} = 2 n, \dots$, and each $X_i$ is a spherical harmonic of certain degree. Then for a fixed $i \geq 0$ and any $h \in C^2 (\R_+)$, we have
$$
L(h X_i) = (L_i h) X_i,
$$
where
\begin{equation}\label{eq:yp-Li}
L_i := - \frac{ {\rm d}^2 }{ {\rm d} t^2 } + \lambda_i + \frac{(n - 2)^2}{4} - \frac{n + 2}{n - 2} K(0) \xi^\frac{4}{n - 2}.
\end{equation}
Define {\it the index set} associated with a Fowler solution $\xi$ as
\begin{equation}\label{indset-01}
\aligned
\mathcal{I} [\xi] := &~ \bigg\{ \sum_{i \geq 1} n_i \rho_i > 0 : n_i \in \N ~ \textmd{with finitely many} ~ n_i > 0 \bigg\} \\
= &~ \bigg\{ \mu_i : 0 < \mu_1 < \mu_2 < \cdots < \mu_i < \mu_{i + 1} < \cdots,~ i \in \N_+ \bigg\}.
\endaligned
\end{equation}
Here $\{ \rho_i \}_{i \geq 1}$ is the sequence of positive constants defined as in Lemmas \ref{lem:Han21} and \ref{lem:Han22} in Section \ref{Preli}. Note that $\mu_1 = \rho_1 = 1$ and $\mu_2 = \min \{ 2, \rho_{n + 1} \}$.

\begin{theorem}\label{thm:csc}
Suppose that $K \in C^1 (B_1)$ is a positive function and $u \in C^2 (B_1 \setminus \{ 0 \})$ is a solution of \eqref{eq:csc} with a non-removable singularity.
\begin{enumerate}[label = \rm(\roman*)]
\item If $n \in \{ 3, 4 \}$ or $\nabla K(0) \neq 0$, then there exists a radial solution $u_0$ of \eqref{eq:u0} such that
\begin{equation}\label{eq:optimal}
u(x) = u_0 (|x|) ( 1 + O(- |x| \ln |x|) ) ~~~~~~ \textmd{as} ~ |x| \to 0^+.
\end{equation}

\item If $n \geq 5$, $\nabla K(0) = 0$, and $K$ satisfies Hypothesis H or $K$ satisfies
$$
0 < c_1 |x|^\frac{n - 4}{2} \leq | \nabla K(x) | \leq c_2 |x|^\frac{n - 4}{2} ~~~~~~ \textmd{for} ~ x \in B_1 \setminus \{ 0 \},
$$
then there exists a radial solution $u_0$ of \eqref{eq:u0} such that
\begin{equation}\label{24eq=poiy}
u(x) = u_0 (|x|) \bigg( 1 + \sum_{i = 1}^m \sum_{j = 0}^{i - 1} c_{ij} (x) |x|^{\mu_i} (- \ln |x|)^j + O \Big( |x|^\frac{n - 2}{2} (- \ln |x|)^m \Big) \bigg) ~~ \textmd{as} ~ |x| \to 0^+,
\end{equation}
where $m$ is the largest positive integer such that $1 = \mu_1 < \cdots < \mu_m < (n - 2)/2$, $\mu_1, \dots, \mu_m$ are the first $m$ elements in the index set $\mathcal{I} [e^{(2 - n) t/2} u_0 (e^{- t})]$, and $c_{ij}$ is a bounded smooth function on $B_{1/2} \setminus \{ 0 \}$ for each $i, j$ in the summation.
\end{enumerate}
\end{theorem}

Theorem \ref{thm:csc} improves the results in Chen-Lin \cite{CL97,CL99a} and Taliaferro-Zhang \cite{TZ06}, where the authors showed 
$$
u(x) = u_0 (|x|) ( 1 + O(|x|^\gamma) ) ~~~~~~ \textmd{as} ~ |x| \to 0^+
$$
for some $\gamma \in (0, 1)$. Moreover, when $n \geq 6$, we would give the first and second order expansion terms in \eqref{24eq=poiy} precisely (see Theorem \ref{thm:csc-expans-2} below). In dimension $3$ or $4$, we can also obtain a higher-order expansion by assuming that $K$ has better flatness near the origin.

\begin{remark}
The convergence in \eqref{eq:optimal} for $\nabla K(0) \neq 0$ should be optimal. In particular, for $n = 4$ we have the following example. Let
$$
u_0 (|x|) := |x|^{- 1}, ~~~~~~ u(x) := u_0 (x) ( 1 + (x_1 + \cdots + x_4) (1 - \ln |x|)/16 ),
$$
$$
K(x) :=
\begin{cases}
\frac{ 1 + 3 (x_1 + \cdots + x_4) (1 - \ln |x|)/16 + (x_1 + \cdots + x_4)/8 }{ ( 1 + (x_1 + \cdots + x_4) (1 - \ln |x|)/16 )^3 } ~~~ & \textmd{if} ~ x \in B_1 \setminus \{ 0 \}, \\
1 ~~~~~~ & \textmd{if} ~ x = 0.
\end{cases}
$$
Then $u_0$ and $u$ are $C^2$ positive solutions of \eqref{eq:u0} and \eqref{eq:csc} respectively. It is elementary to check that $K \in C^1 (B_1)$ and $\nabla K(0) = \big( \frac{1}{8}, \frac{1}{8}, \frac{1}{8}, \frac{1}{8} \big) \in \R^4$.
\end{remark}

\begin{theorem}\label{thm:csc-expans}
Let $K \in C^1 (B_1)$ be a positive function satisfying
$$
K(x) = K(0) + O(|x|^\beta) ~~~~~~ \textmd{for some} ~ \beta \geq 1 ~ \textmd{near the origin}.
$$
Suppose that $v$ is a positive solution of \eqref{eq:v} in $\R_+ \times \Sp^{n - 1}$ and $\xi$ is a positive periodic solution of \eqref{eq:xi} satisfying
\begin{equation}\label{eq:v2xi}
v(t, \theta) \to \xi (t) ~~~~~~ \textmd{as} ~ t \to \infty ~ \textmd{uniformly for} ~ \theta \in \Sp^{n - 1}.
\end{equation}
\begin{enumerate}[label = \rm(\roman*)]
\item If $\beta = 1$, then there exists a constant $C > 0$ such that for any $(t, \theta) \in (1, \infty) \times \Sp^{n - 1}$,
$$
| v(t, \theta) - \xi (t) | \leq C t e^{- t}.
$$

\item If $\beta > 1$, then there exists a spherical harmonic $Y$ of degree $1$ and a constant $C > 0$ such that for any $(t, \theta) \in (1, \infty) \times \Sp^{n - 1}$,
$$
\bigg| v(t, \theta) - \xi (t) - e^{- t} \bigg( \frac{n - 2}{2} \xi (t) - \xi' (t) \bigg) Y(\theta) \bigg| \leq C e^{- \gamma t},
$$
where $\gamma \in (1, 2)$ is a constant. Furthermore, there exists a constant $C > 0$ such that for any $(t, \theta) \in (1, \infty) \times \Sp^{n - 1}$,
\begin{equation}\label{eq:csc-expans}
\bigg| v(t, \theta) - \xi (t) - \sum_{i = 1}^m \sum_{j = 0}^{i - 1} c_{ij} (t, \theta) t^j e^{- \mu_i t} \bigg| \leq
\bigg\{
\aligned
& C e^{- \beta t} ~ & \textmd{when} & ~ \mu_m < \beta < \mu_{m + 1}, \\
& C t^m e^{- \beta t} ~ & \textmd{when} & ~ \beta = \mu_{m + 1},
\endaligned
\end{equation}
where $\{ \mu_m \}_{m \geq 1}$ is defined as in \eqref{indset-01} and $c_{ij}$ is a bounded smooth function on $(1, \infty) \times \Sp^{n - 1}$ for each $i, j$ in the summation.
\end{enumerate}
\end{theorem}

When $n \geq 6$, we could show that $\mu_2 = 2$ (see Corollary \ref{App-cor} in Appendix \ref{AppendixA}). In this case, we establish the following refined second-order expansion.

\begin{theorem}\label{thm:csc-expans-2}
Assume $n \geq 6$. Let $K \in C^1 (B_1)$ be a positive function satisfying
$$
K(x) = K(0) + O(|x|^\beta) ~~~~~~ \textmd{for some} ~ \beta > 2 ~ \textmd{near the origin}.
$$
Suppose that $v$ is a positive solution of \eqref{eq:v} in $\R_+ \times \Sp^{n - 1}$ and $\xi$ is a positive periodic solution of \eqref{eq:xi} satisfying \eqref{eq:v2xi}. Then there exists a spherical harmonic $Y = Y(\theta)$ of degree $1$ and a constant $C > 0$ such that for any $(t, \theta) \in (1, \infty) \times \Sp^{n - 1}$,
$$
| v(t, \theta) - \xi (t) - \xi_1 (t, \theta) - \xi_2 (t, \theta) | \leq C e^{- \gamma t},
$$
where
$$
\xi_1 (t, \theta) := e^{- t} \bigg( \frac{n - 2}{2} \xi (t) - \xi' (t) \bigg) Y(\theta),
$$
$$
\xi_2 (t, \theta) := e^{- 2 t} \bigg[ - \frac{1}{2} K(0) \xi (t)^\frac{n + 2}{n - 2} Y(\theta)^2 + \bigg( - \frac{n - 2}{8} \xi (t) + \frac{1}{4} \xi' (t) \bigg) \Delta_\theta (Y^2) \bigg],
$$
and $\gamma \in (2, 3)$ is a constant.
\end{theorem}

Now we consider the existence of solutions of \eqref{eq:v} which satisfy the asymptotic expansion \eqref{eq:csc-expans}. For a nonnegative function $f \in C^2 (\R_+ \times \Sp^{n - 1})$, define
\begin{equation}\label{eq:csc-M}
\mathcal{M} (f) := - f_{tt} - \Delta_\theta f + \frac{(n - 2)^2}{4} f - K f^\frac{n + 2}{n - 2}.
\end{equation}

\begin{theorem}\label{thm:csc-existv1}
Let $\xi$ be a positive periodic solution of \eqref{eq:xi} with the index set $\mathcal{I} [\xi]$, and let $K \in C^1 (B_1)$ be a positive function satisfying
$$
| \nabla K(x) | = O(|x|^{\beta - 1}) ~~~~~~ \textmd{for some} ~ \beta > 1 ~ \textmd{and} ~ \beta \notin \mathcal{I} [\xi] ~ \textmd{near the origin}.
$$
Suppose that $\widehat{v} \in C^{2, \alpha} (\R_+ \times \Sp^{n - 1})$ satisfies
\begin{equation}\label{eq:hatv-xi}
| (\widehat{v} - \xi) (t, \theta) | + | \nabla (\widehat{v} - \xi) (t, \theta) | \to 0 ~~~~~~ \textmd{as} ~ t \to \infty ~ \textmd{uniformly for} ~ \theta \in \Sp^{n - 1},
\end{equation}
and, for any $(t, \theta) \in (1, \infty) \times \Sp^{n - 1}$,
\begin{equation}\label{eq:Mhatv}
| \mathcal{M} (\widehat{v}) (t, \theta) | + | \nabla ( \mathcal{M} (\widehat{v}) ) (t, \theta) | \leq C e^{- \beta t}
\end{equation}
for some positive constant $C$. Then there exist constants $t_0 > 1, C > 0$ and a positive solution $v \in C^{2, \alpha} ([t_0, \infty) \times \Sp^{n - 1})$ of \eqref{eq:v} such that for any $(t, \theta) \in (t_0, \infty) \times \Sp^{n - 1}$,
$$
| (v - \widehat{v}) (t, \theta) | \leq C e^{- \beta t}.
$$
\end{theorem}

\begin{remark}
The choice of function $\widehat{v}$ in Theorem \ref{thm:csc-existv1} is flexible. If $\widehat{v}$ satisfies \eqref{eq:hatv-xi}, the condition \eqref{eq:Mhatv} is equivalent to
$$
| \mathcal{M}_1 (\widehat{v}) (t, \theta) | + | \nabla ( \mathcal{M}_1 (\widehat{v}) ) (t, \theta) | \leq C e^{- \beta t},
$$
where
$$
\mathcal{M}_1 (\widehat{v}) := - \widehat{v}_{tt} - \Delta_\theta \widehat{v} + \frac{(n - 2)^2}{4} \widehat{v} - K(0) \widehat{v}^\frac{n + 2}{n - 2}.
$$
Han and Li \cite{HL} provided a general procedure to construct such approximate solution $\widehat{v}$ (see Proposition 3.3 there).
\end{remark}

When taking $\widehat{v} = \xi$ in Theorem \ref{thm:csc-existv1}, we have the following result which holds for all $\beta > 1$.

\begin{theorem}\label{thm:csc-existv2}
Let $K \in C^1 (B_1)$ be a positive function satisfying
$$
| \nabla K(x) | = O(|x|^{\beta - 1}) ~~~~~~ \textmd{for some} ~ \beta > 1 ~ \textmd{near the origin}.
$$
Suppose that $\xi$ is a positive periodic solution of \eqref{eq:xi}. Then there exist two constants $t_0 > 1, C > 0$ and a positive solution $v \in C^{2, \alpha} ([t_0, \infty) \times \Sp^{n - 1})$ of \eqref{eq:v} such that for any $(t, \theta) \in (t_0, \infty) \times \Sp^{n - 1}$,
$$
| v(t, \theta) - \xi (t) | \leq
\bigg\{
\aligned
& C e^{- \beta t} ~~~~~~ & \textmd{if} ~ \beta \notin \{ \rho_i \}_{i \geq 1}, \\
& C t e^{- \beta t} ~~~~~~ & \textmd{if} ~ \beta \in \{ \rho_i \}_{i \geq 1},
\endaligned
$$
where $\{ \rho_i \}_{i \geq 1}$ is the sequence of positive constants defined as in Lemmas \ref{lem:Han21} and \ref{lem:Han22}.
\end{theorem}

In $x$-coordinates, the above result can be stated as
\begin{corollary}\label{thm:csc-exist}
Let $K \in C^1 (B_1)$ be a positive function satisfying
$$
| \nabla K(x) | = O(|x|^{\beta - 1}) ~~~~~~ \textmd{for some} ~ \beta > 1 ~ \textmd{near the origin}.
$$
Let $\xi$ be a positive periodic solution of \eqref{eq:xi}. Then there exist two constants $0 < R < 1, C > 0$ and a positive solution $u$ of \eqref{eq:csc} in $B_R \setminus \{ 0 \}$ such that for any $x \in B_R \setminus \{ 0 \}$,
$$
\bigg| u(x) - |x|^{ - \frac{n - 2}{2} } \xi (- \ln |x|) \bigg| \leq
\bigg\{
\aligned
& C |x|^{\beta - (n - 2)/2} ~~~~~~ & \textmd{if} ~ \beta \notin \{ \rho_i \}_{i \geq 1}, \\
& C |x|^{\beta - (n - 2)/2} (- \ln |x|) ~~~~~~ & \textmd{if} ~ \beta \in \{ \rho_i \}_{i \geq 1},
\endaligned
$$
where $\{ \rho_i \}_{i \geq 1}$ is the sequence of positive constants defined as in Lemmas \ref{lem:Han21} and \ref{lem:Han22}.
\end{corollary}

In \cite{CKN}, Caffarelli, Kohn and Nirenberg established the following inequality
\begin{equation}\label{eq:ckn}
\bigg( \int_{\R^n} |x|^{- b p} |u|^p {\rm d} x \bigg)^\frac{2}{p} \leq C_{a, b} \int_{\R^n} |x|^{- 2 a} | \nabla u |^2 {\rm d} x ~~~~~~ \textmd{for all} ~ u \in C_c^\infty (\R^n),
\end{equation}
where $n \geq 3$,
$$
- \infty < a < \frac{n - 2}{2}, ~~~~~~ a \leq b \leq a + 1 ~~~~~~ \textmd{and} ~~~~~~ p = \frac{2 n}{n - 2 + 2 (b - a)}.
$$
There have been a lot of works concerning the best constant, existence (or nonexistence) and symmetry (or symmetry-breaking) of extremal functions of \eqref{eq:ckn}; see \cite{CW,CC,DEL,LW,FS} and the references therein. Any extremal function of \eqref{eq:ckn} satisfies, with a proper normalization, the Euler-Lagrange equation
$$
- \,{\rm div} (|x|^{- 2 a} \nabla u) = |x|^{- b p} u^{p - 1} ~~~~~~ \textmd{in} ~ \R^n.
$$
Hsia, Lin and Wang in \cite{HLW} studied the asymptotic behavior of solutions of
\begin{equation}\label{eq:cknel}
- \,{\rm div} (|x|^{- 2 a} \nabla u) = |x|^{- b p} u^{p - 1}, ~~~ u > 0 ~~~ \textmd{in} ~ B_1 \setminus \{ 0 \}
\end{equation}
with $n \geq 3$,
\begin{equation}\label{eq:range}
0 \leq a < \frac{n - 2}{2}, ~~~~~~ a \leq b < a + 1 ~~~~~~ \textmd{and} ~~~~~~ p = \frac{2 n}{n - 2 + 2 (b - a)}.
\end{equation}
They showed that every singular solution $u \in C^2 (B_1 \setminus \{ 0 \})$ of \eqref{eq:cknel} is asymptotic to a global singular solution $u_0$, i.e.,
\begin{equation}\label{eq1=lw}
u(x) = u_0 (x) (1 + o(1)) ~~~~~~ \textmd{as} ~ |x| \to 0^+,
\end{equation}
where $u_0 (x) =u_0 (|x|) \in C^2 (\R \setminus \{ 0 \})$ is a global singular solution of \eqref{eq:cknel}. Guo, Li and Wan \cite{GLW} further established an arbitrary-order expansion in \eqref{eq1=lw} for
$$
0 \leq a < \frac{n - 2}{2}, ~~~~~~ a = b ~~~~~~ \textmd{and} ~~~~~~ p = \frac{2 n}{n - 2}.
$$

As the second part of this paper, we will establish the arbitrary-order expansion of singular solutions of \eqref{eq:cknel} under the assumption \eqref{eq:range}. Set $t = - \ln |x|$, $\theta = x/|x| \in \Sp^{n - 1}$ and
$$
w(t, \theta) = e^{ - \frac{n - 2 a - 2}{2} t } u(e^{- t} \theta) ~~~~~~ \textmd{for} ~ (t, \theta) \in \R_+ \times \Sp^{n - 1}.
$$
Then $w$ solves the equation
\begin{equation}\label{eq:w}
- w_{tt} - \Delta_\theta w + \frac{(n - 2 a - 2)^2}{4} w = w^{p - 1} ~~~~~~ \textmd{in} ~ \R_+ \times \Sp^{n - 1}.
\end{equation}
Let $u_0 (x) = u_0 (|x|)$ be a global singular solution (i.e., Fowler solution) of \eqref{eq:cknel} and $\zeta (t) = e^{ (2 + 2 a - n) t/2 } u_0 (e^{- t})$. Then $\zeta$ is a positive periodic solution of
\begin{equation}\label{eq:zeta}
- \zeta'' + \frac{(n - 2 a - 2)^2}{4} \zeta = \zeta^{p - 1} ~~~~~~ \textmd{in} ~ \R.
\end{equation}
Define the linearized operator of \eqref{eq:w} at a Fowler solution $\zeta$
\begin{equation}\label{eq:ckn-L}
\mathcal{L} := - \frac{\partial^2}{\partial t^2} - \Delta_\theta + \frac{(n - 2 a - 2)^2}{4} - (p - 1) \zeta^{p - 2},
\end{equation}
and the {\it index set} associated with $\zeta$ as
\begin{equation}\label{indset-02}
\aligned
\mathcal{S} [\zeta] := &~ \bigg\{ \sum_{i \geq 1} n_i \sigma_i > 0 : n_i \in \N ~ \textmd{with finitely many} ~ n_i > 0 \bigg\} \\
= &~ \bigg\{ \nu_i : 0 < \nu_1 < \nu_2 < \cdots < \nu_i < \nu_{i + 1} < \cdots,~ i \in \N_+ \bigg\}.
\endaligned
\end{equation}
Here $\{ \sigma_i \}_{i \geq 1}$ is the sequence of positive constants given in Lemmas \ref{lem:ckn-Lc} and \ref{lem:ckn-Ln} in Appendix \ref{AppendixB}. Note that $\nu_1 = \sigma_1$ and $\nu_2 = \min \{ 2 \sigma_1, \sigma_{n + 1} \}$.

\begin{theorem}\label{thm:ckn-expans}
Assume that \eqref{eq:range} holds. Suppse that $w$ is a positive solution of \eqref{eq:w} in $\R_+ \times \Sp^{n - 1}$ and $\zeta$ is a positive periodic solution of \eqref{eq:zeta} with the index set $\mathcal{S} [\zeta]$ satisfying
\begin{equation}\label{eq:w2zeta}
w(t, \theta) \to \zeta (t) ~~~~~~ \textmd{as} ~ t \to \infty ~ \textmd{uniformly for} ~ \theta \in \Sp^{n - 1}.
\end{equation}
Then for any positive integer $m$ and any $(t, \theta) \in (1, \infty) \times \Sp^{n - 1}$,
\begin{equation}\label{eq:ckn-expans}
\bigg| w(t, \theta) - \zeta (t) - \sum_{i = 1}^m \sum_{j = 0}^{i - 1} c_{ij} (t, \theta) t^j e^{- \nu_i t} \bigg| \leq C t^m e^{ - \nu_{m + 1} t },
\end{equation}
where $C$ is a positive constant depending only on $a$, $b$, $\zeta$ and $m$, and $c_{ij}$ is a bounded smooth function on $(1, \infty) \times \Sp^{n - 1}$ for each $i, j$ in the summation.
\end{theorem}

Next we establish the existence of solutions of \eqref{eq:w} which satisfy the asymptotic expansion \eqref{eq:ckn-expans}. For a nonnegative function $f \in C^2 (\R_+ \times \Sp^{n - 1})$, define
\begin{equation}\label{eq:ckn-N}
\mathcal{N} (f) := - f_{tt} - \Delta_\theta f + \frac{(n - 2 a - 2)^2}{4} f - f^{p - 1}.
\end{equation}

\begin{theorem}\label{thm:ckn-existw1}
Assume that \eqref{eq:range} holds. Let $\zeta$ be a positive periodic solution of \eqref{eq:zeta} with the index set $\mathcal{S} [\zeta]$. Choose any $\nu > \sigma_1$ with $\nu \notin \mathcal{S} [\zeta]$. Suppose that $\widehat{w} \in C^{2, \alpha} (\R_+ \times \Sp^{n - 1})$ satisfies
\begin{equation}\label{eq:hatw-zeta}
| (\widehat{w} - \zeta) (t, \theta) | + | \nabla (\widehat{w} - \zeta) (t, \theta) | \to 0 ~~~~~~ \textmd{as} ~ t \to \infty ~ \textmd{uniformly for} ~ \theta \in \Sp^{n - 1},
\end{equation}
and, for any $(t, \theta) \in (1, \infty) \times \Sp^{n - 1}$,
\begin{equation}\label{eq:Nhatw}
| \mathcal{N} (\widehat{w}) (t, \theta) | + | \nabla ( \mathcal{N} (\widehat{w}) ) (t, \theta) | \leq C e^{- \nu t}
\end{equation}
for some positive constant $C$. Then there exist two constants $t_0 > 1, C > 0$ and a positive solution $w \in C^{2, \alpha} ([t_0, \infty) \times \Sp^{n - 1})$ of \eqref{eq:w} such that
$$
| (w - \widehat{w}) (t, \theta) | \leq C e^{- \nu t} ~~~~~~ \textmd{for any} ~ (t, \theta) \in (t_0, \infty) \times \Sp^{n - 1}.
$$
\end{theorem}

This paper is organized as follows. In Section \ref{Preli}, we recall several properties of operators $L$ and $L_i$ defined in \eqref{eq:yp-L} and \eqref{eq:yp-Li}. In Section \ref{Sec=3}, we show the asymptotic expansions in Theorems \ref{thm:csc}, \ref{thm:csc-expans} and \ref{thm:csc-expans-2}, and establish the existence results in Theorems \ref{thm:csc-existv1} and \ref{thm:csc-existv2}. In Section \ref{Sec=04}, we prove the expansion in Theorem \ref{thm:ckn-expans} and the existence in Theorem \ref{thm:ckn-existw1}. In Appendix \ref{Sec=05}, we give a lower bound for $\rho_i$ appearing in Lemmas \ref{lem:Han21} and \ref{lem:Han22}, and analyze some properties of the operator $\mathcal{L}$ defined in \eqref{eq:ckn-L}, which are used to prove the expansion in Theorem \ref{thm:ckn-expans}.

\section{Preliminaries}\label{Preli}

Here we recall several properties of operators $L$ and $L_i$ defined in \eqref{eq:yp-L} and \eqref{eq:yp-Li}, respectively. We refer to \cite{HLL,KMPS,MP,MPU} for more details.

\begin{lemma}\label{lem:Han21}
Let $\xi$ be the positive constant solution of \eqref{eq:xi}.
\begin{enumerate}[label = \rm(\roman*)]
\item For $i = 0$, ${\rm Ker} (L_0)$ has a basis $\cos (\sqrt{n - 2} t)$ and $\sin (\sqrt{n - 2} t)$.

\item There exists an increasing sequence of positive constants $\{ \rho_i \}_{i \geq 1}$, divergent to $\infty$, such that for any $i \geq 1$, ${\rm Ker} (L_i)$ has a basis $e^{- \rho_i t}$ and $e^{\rho_i t}$. Moreover, $\rho_1 = \cdots = \rho_n = 1$.
\end{enumerate}
\end{lemma}

\begin{lemma}\label{lem:Han22}
Let $\xi$ be a positive nonconstant periodic solution of \eqref{eq:xi}.
\begin{enumerate}[label = \rm(\roman*)]
\item For $i = 0$, ${\rm Ker} (L_0)$ has a basis $p_0^+$ and $p_0^- + c t p_0^+$, for some smooth periodic functions $p_0^+$ and $p_0^-$ on $\R$, and some constant $c$.

\item There exists an increasing sequence of positive constants $\{ \rho_i \}_{i \geq 1}$, divergent to $\infty$, such that for any $i \geq 1$, ${\rm Ker} (L_i)$ has a basis $e^{- \rho_i t} p_i^+$ and $e^{\rho_i t} p_i^-$, for some smooth periodic functions $p_i^+$ and $p_i^-$ on $\R$. Moreover, for $i = 1, \dots, n$,
$$
\rho_i = 1, ~~~~~~ p_i^+ = \frac{n - 2}{2} \xi - \xi' ~~~~~~ \textmd{and} ~~~~~~ p_i^- = \frac{n - 2}{2} \xi + \xi'.
$$
\end{enumerate}

In addition, all periodic functions in $\rm(i)$ and $\rm(ii)$ have the same period as $\xi$.
\end{lemma}

We also need the following decay estimates whose proofs can be found in Han-Li-Li \cite{HLL}.

\begin{lemma}\label{lem:HanA8}
Let $\gamma \geq 1$ be a constant, $m \geq 0$ be an integer, and $f$ be a continuous function in $[1, \infty) \times \Sp^{n - 1}$ satisfying
$$
| f(t, \theta) | \leq C t^m e^{- \gamma t}, ~~~~~~ \forall ~ (t, \theta) \in [1, \infty) \times \Sp^{n - 1}.
$$
Let $\varphi$ be a solution of $L \varphi = f$ in $(1, \infty) \times \Sp^{n - 1}$ such that $\varphi (t, \theta) \to 0$ as $t \to \infty$ uniformly for $\theta \in \Sp^{n - 1}$.
\begin{enumerate}[label = \rm(\roman*)]
\item If $\gamma = \rho_1 = 1$, then for any $(t, \theta) \in (1, \infty) \times \Sp^{n - 1}$,
$$
| \varphi (t, \theta) | \leq C t^{m + 1} e^{- t}.
$$

\item If $\rho_l < \gamma \leq \rho_{l + 1}$ for some positive integer $l$, then for any $(t, \theta) \in (1, \infty) \times \Sp^{n - 1}$,
$$
\bigg| \varphi (t, \theta) - \sum_{i = 1}^l c_i e^{- \rho_i t} p_i^+ (t) X_i (\theta) \bigg| \leq
\bigg\{
\aligned
& C t^m e^{- \gamma t} ~~~~~~ & \textmd{if} & ~ \rho_l < \gamma < \rho_{l + 1}, \\
& C t^{m + 1} e^{- \gamma t} ~~~~~~ & \textmd{if} & ~ \gamma = \rho_{l + 1},
\endaligned
$$
with some constants $c_i$ for $i = 1, \dots, l$.
\end{enumerate}
\end{lemma}

\section{Conformal scalar curvature equation}\label{Sec=3}

In this section, we first show the asymptotic expansion in Theorem \ref{thm:csc-expans} using analysis of the linearized operators at the Fowler solutions, along the lines of methods in Korevaar-Mazzeo-Pacard-Schoen \cite{KMPS} and Han-Li-Li \cite{HLL}. The main difficulty here is that $K$ is a positive function instead of a positive constant. Furthermore, we establish the refined second-order asymptotic expansion as stated in Theorem \ref{thm:csc-expans-2}. Then we apply the contraction mapping principle to obtain the existence results in Theorems \ref{thm:csc-existv1} and \ref{thm:csc-existv2}, which is inspired by Han-Li \cite{HL} and Mazzeo-Pacard \cite{MP}.

\subsection{Asymptotic expansion}

\begin{proof}[Proof of Theorem \ref{thm:csc-expans}] Throughout the proof, we always assume that $\xi$ is a positive nonconstant periodic solution of \eqref{eq:xi}. The proof is similar for the case when $\xi$ is the positive constant solution. Without lost of generality, we may suppose $K(0) = 1$.

Define
$$
\varphi (t, \theta) = v(t, \theta) - \xi (t) ~~~~~~ \textmd{for} ~ (t, \theta) \in \R_+ \times \Sp^{n - 1}.
$$
Then by \eqref{eq:v2xi}, we know
$$
\varphi (t, \theta) \to 0 ~~~~~~ \textmd{as} ~ t \to \infty ~ \textmd{uniformly for} ~ \theta \in \Sp^{n - 1}.
$$
From the assumption $K(e^{- t} \theta) - 1 = O(e^{- \beta t})$, we have
\begin{equation}\label{eq:Lvphi}
\aligned
L \varphi & = K (\xi + \varphi)^\frac{n + 2}{n - 2} - \xi^\frac{n + 2}{n - 2} - \frac{n + 2}{n - 2} \xi^\frac{4}{n - 2} \varphi \\
& = \bigg[ (\xi + \varphi)^\frac{n + 2}{n - 2} - \xi^\frac{n + 2}{n - 2} - \frac{n + 2}{n - 2} \xi^\frac{4}{n - 2} \varphi \bigg] + (K - 1) (\xi + \varphi)^\frac{n + 2}{n - 2} \\
& =: F(\varphi) + O(e^{- \beta t}).
\endaligned
\end{equation}
We decompose the index set $\mathcal{I} [\xi]$ into two parts. Let
$$
\mathcal{I}_1 = \{ \rho_i : i \geq 1 \}
$$
and
$$
\mathcal{I}_2 = \bigg\{ \sum_{i = 1}^k n_i \rho_i > 0 : n_i \in \N ~ \textmd{and} ~ \sum_{i = 1}^k n_i \geq 2 \bigg\}.
$$
Denote the set $\mathcal{I}_2$ by a strictly increasing sequence $\{ \widetilde{\rho}_i \}_{i \geq 1}$. Obviously, $\widetilde{\rho}_1 = 2$.

First we prove \eqref{eq:csc-expans} for the simple case: $\mathcal{I}_1 \cap \mathcal{I}_2 = \varnothing$. Set
$$
\aligned
1 = \rho_1 \leq \cdots \leq \rho_{k_1} < \widetilde{\rho}_1 < \cdots & < \widetilde{\rho}_{l_1} \\
& < \rho_{k_1 + 1} \leq \cdots \leq \rho_{k_2} < \widetilde{\rho}_{l_1 + 1} < \cdots.
\endaligned
$$

{\bf Case 1:} $\beta = \rho_1 = 1$. It follows from \cite[Theorem 1]{TZ06} that $\varphi = O(e^{- 3 t/4})$. Since $| F(\varphi) | \leq C \varphi^2$, we have $F(\varphi) = O(e^{- 3 t/2})$, and thus $L \varphi = O(e^{- t})$. By Lemma \ref{lem:HanA8}, we know that for any $(t, \theta) \in (1, \infty) \times \Sp^{n - 1}$,
$$
| \varphi (t, \theta) | \leq C t e^{- t}.
$$

{\bf Case 2:} $\rho_m < \beta \leq \rho_{m + 1}$ for some $m \in \{ 1, 2, \dots, k_1 - 1 \}$. In this case, we have $1 = \rho_1 < \beta < \widetilde{\rho}_1 = 2$. Again, by \cite[Theorem 1]{TZ06} we obtain $\varphi = O(e^{- \beta t/2})$ and $L \varphi = O(e^{- \beta t})$. By Lemma \ref{lem:HanA8} we can get
$$
\bigg| \varphi (t, \theta) - \sum_{i = 1}^m c_i p_i^+ (t) X_i (\theta) e^{- \rho_i t} \bigg| \leq
\bigg\{
\aligned
& C e^{- \beta t} ~~~~~~ & \textmd{if} & ~ \rho_m < \beta < \rho_{m + 1}, \\
& C t e^{- \beta t} ~~~~~~ & \textmd{if} & ~ \beta = \rho_{m + 1},
\endaligned
$$
with some constants $c_i$ for $i = 1, \dots, m$.

{\bf Case 3:} $\rho_{k_1} < \beta \leq \widetilde{\rho}_1 = 2$. From Case 2, we know that $\varphi = O(e^{- \rho_1 t})$. Therefore, $F(\varphi) = O(e^{ - \widetilde{\rho}_1 t })$ and $L \varphi = O(e^{- \beta t})$. By Lemma \ref{lem:HanA8} we have
$$
\bigg| \varphi (t, \theta) - \sum_{i = 1}^{k_1} c_i p_i^+ (t) X_i (\theta) e^{- \rho_i t} \bigg| \leq C e^{- \beta t},
$$
for some constants $c_i$, $i = 1, \dots, k_1$.

From now on, we denote
\begin{equation}\label{eq:eta1}
\eta_1 (t, \theta) := \sum_{i = 1}^{k_1} c_i p_i^+ (t) X_i (\theta) e^{- \rho_i t},
\end{equation}
and
$$
\varphi_1 (t, \theta) := \varphi (t, \theta) - \eta_1 (t, \theta).
$$
Then we have $L \eta_1 = 0$, and by \eqref{eq:Lvphi},
\begin{equation}\label{eq:Lvphi1}
L \varphi_1 = F(\varphi) + O(e^{- \beta t}).
\end{equation}
For each $\widetilde{\rho}_i \in \mathcal{I}_2$, we consider nonnegative integers $n_1, \dots, n_k$ such that
\begin{equation}\label{eq:trhoi}
n_1 \rho_1 + \cdots + n_k \rho_k = \widetilde{\rho}_i ~~~~~~ \textmd{and} ~~~~~~ n_1 + \cdots + n_k \geq 2.
\end{equation}
Notice that there are only finitely many collections of nonnegative integers $n_1, \dots, n_k$ satisfying \eqref{eq:trhoi}. Set
$$
\widetilde{K}_i = \max \{ n_1 \deg (X_1) + \cdots + n_k \deg (X_k) : n_1, \dots, n_k \in \N ~ \textmd{satisfy \eqref{eq:trhoi}} \},
$$
and
$$
\widetilde{M}_i = \max \{ m : \deg (X_m) \leq \widetilde{K}_i \}.
$$

{\bf Case 4:} $\widetilde{\rho}_m < \beta \leq \widetilde{\rho}_{m + 1}$ for some $m \in \{ 1, 2, \dots, l_1 - 1 \}$. We claim that there exists $\widetilde{\eta}_1$ having the form of
\begin{equation}\label{eq:teta1m}
\widetilde{\eta}_1 (t, \theta) = \sum_{i = 1}^m \sum_{j = 0}^{ \widetilde{M}_i } \widetilde{c}_{ij} (t) X_j (\theta) e^{ - \widetilde{\rho}_i t }
\end{equation}
with $\widetilde{c}_{ij}$ being smooth periodic functions such that for
\begin{equation}\label{eq:tvphi1m}
\widetilde{\varphi}_1 := \varphi_1 - \widetilde{\eta}_1,
\end{equation}
we have $L \widetilde{\varphi}_1 = O(e^{- \beta t})$. Take some function $\widetilde{\eta}_1$ to be determined later, and define $\widetilde{\varphi}_1$ as \eqref{eq:tvphi1m}. By \eqref{eq:Lvphi1} we get
\begin{equation}\label{eq:Ltvphi1m}
L \widetilde{\varphi}_1 = L \varphi_1 - L \widetilde{\eta}_1 = F(\varphi) - L \widetilde{\eta}_1 + O(e^{- \beta t}).
\end{equation}
We will construct a proper $\widetilde{\eta}_1$ such that
$$
F(\varphi) - L \widetilde{\eta}_1 = O(e^{- \beta t}).
$$
Now we expand $F(\varphi)$. From Case 3, we know that $\varphi = O(e^{- \rho_1 t})$, $\eta_1 = O(e^{- \rho_1 t})$ and $\varphi_1 = O(e^{- 2 \rho_1 t})$. Therefore,
\begin{equation}\label{eq:Feta1}
F(\varphi) = \sum_{i = 2}^{[\beta]} a_i \varphi^i + O(e^{- \beta t}) = \sum_{i = 2}^{[\beta]} a_i (\eta_1 + \varphi_1)^i + O(e^{- \beta t}),
\end{equation}
where $[\cdot]$ is the floor function and $a_i = a_i (t)$ is a smooth periodic function (with the same period as $\xi$). For the lower order terms involving $\varphi_1$, we have
\begin{equation}\label{eq:lotvphi1}
\eta_1 \varphi_1 = O(e^{- 3 \rho_1 t}) ~~~~~~ \textmd{and} ~~~~~~ \varphi_1^2 = O(e^{- 4 \rho_1 t}).
\end{equation}
For the term involving $\eta_1$ only, by \eqref{eq:eta1} and \cite[Lemma 2.4]{HLL} we have
\begin{equation}\label{eq:sumeta1}
\aligned
\sum_{i = 2}^{[\beta]} a_i \eta_1^i & = \sum_{n_1 + \cdots + n_k \geq 2}^{[\beta]} a_{n_1 \cdots n_k} (t) X_1^{n_1} \cdots X_k^{n_k} e^{ - (n_1 \rho_1 + \cdots + n_k \rho_k) t } + O(e^{- \beta t}) \\
& = \sum_{i = 1}^m \sum_{j = 0}^{ \widetilde{M}_i } a_{ij} (t) X_j (\theta) e^{ - \widetilde{\rho}_i t } + O(e^{- \beta t}),
\endaligned
\end{equation}
where $a_{ij}$ is a smooth periodic function (with the same period as $\xi$). We discuss this in two subcases.

{\it Subcase 1:} $\widetilde{\rho}_m < 3 \rho_1$. Then we have $\beta \leq \widetilde{\rho}_{m + 1} \leq 3 \rho_1$, and thus by \eqref{eq:lotvphi1},
$$
\eta_1 \varphi_1 = O(e^{- \beta t}) ~~~~~~ \textmd{and} ~~~~~~ \varphi_1^2 = O(e^{- \beta t}).
$$
This together with \eqref{eq:Ltvphi1m}-\eqref{eq:sumeta1} implies that
$$
L \widetilde{\varphi}_1 = - L \widetilde{\eta}_1 + \sum_{i = 1}^m \sum_{j = 0}^{ \widetilde{M}_i } a_{ij} (t) X_j (\theta) e^{ - \widetilde{\rho}_i t } + O(e^{- \beta t}).
$$
Since $\mathcal{I}_1 \cap \mathcal{I}_2 = \varnothing$, from \cite[Lemma A.2 and Remark A.5]{HLL} there exists a smooth periodic function $\widetilde{c}_{ij}$ such that
\begin{equation}\label{eq:Ljcijrhoi}
L_j ( \widetilde{c}_{ij} (t) e^{ - \widetilde{\rho}_i t } ) = a_{ij} (t) e^{ - \widetilde{\rho}_i t } ~~~~~~ \textmd{for each} ~ 1 \leq i \leq m, 0 \leq j \leq \widetilde{M}_i,
\end{equation}
where $L_j$ is defined as in \eqref{eq:yp-Li}. Thus we obtain $\widetilde{\eta}_1$ to be the form \eqref{eq:teta1m} and have
$$
L \widetilde{\varphi}_1 = O(e^{- \beta t}).
$$
Recall that $\widetilde{\varphi}_1 = O(e^{ - \widetilde{\rho}_1 t })$, and there is no $\rho_i$ between $\widetilde{\rho}_1$ and $\beta$. It follows from Lemma \ref{lem:HanA8} that $\widetilde{\varphi}_1 = O(e^{- \beta t})$. That is,
$$
| \varphi_1 - \widetilde{\eta}_1 | \leq C e^{- \beta t}.
$$

{\it Subcase 2:} $\widetilde{\rho}_m \geq 3 \rho_1$. Let $n_1$ be the largest integer such that $\widetilde{\rho}_{n_1} < 3 \rho_1$, then $\widetilde{\rho}_{n_1 + 1} = 3 \rho_1 \leq \widetilde{\rho}_m < \beta$. Now we take $\widetilde{\eta}_{11}$ to be the summation \eqref{eq:teta1m} from $1$ to $n_1$, i.e,
$$
\widetilde{\eta}_{11} (t, \theta) = \sum_{i = 1}^{n_1} \sum_{j = 0}^{ \widetilde{M}_i } \widetilde{c}_{ij} (t) X_j (\theta) e^{ - \widetilde{\rho}_i t }.
$$
Then
\begin{equation}\label{eq:Leta11}
L \widetilde{\eta}_{11} = \sum_{i = 1}^{n_1} \sum_{j = 0}^{ \widetilde{M}_i } a_{ij} (t) X_j (\theta) e^{ - \widetilde{\rho}_i t }.
\end{equation}
Let $\widetilde{\varphi}_{11} = \varphi_1 - \widetilde{\eta}_{11}$. By \eqref{eq:Lvphi1} and \eqref{eq:Feta1}-\eqref{eq:Leta11} we have
$$
L \widetilde{\varphi}_{11} = F(\varphi) - L \widetilde{\eta}_{11} + O(e^{- \beta t}) = O(e^{ - \widetilde{\rho}_{n_1 + 1} t }).
$$
Since there is no $\rho_i$ between $\widetilde{\rho}_1$ and $\widetilde{\rho}_{n_1 + 1}$, it follows from $\widetilde{\varphi}_{11} = O(e^{ - \widetilde{\rho}_1 t })$ and Lemma \ref{lem:HanA8} that
$$
\widetilde{\varphi}_{11} = O(e^{ - \widetilde{\rho}_{n_1 + 1} t }).
$$

Now we are in a similar situation as at the beginning of Case 4, with $\widetilde{\rho}_{n_1 + 1} = 3 \rho_1$ replacing $\widetilde{\rho}_1 = 2 \rho_1$. If $\widetilde{\rho}_m < 4 \rho_1$, we proceed as in Step 1. For the sake of completeness, we present the details here. We claim that there exists $\widetilde{\eta}_{111}$ having the form of
\begin{equation}\label{eq:teta111}
\widetilde{\eta}_{111} (t, \theta) = \sum_{i = n_1 + 1}^m \sum_{j = 0}^{ \widetilde{M}_i } \widetilde{c}_{ij} (t) X_j (\theta) e^{ - \widetilde{\rho}_i t }
\end{equation}
such that for
$$
\widetilde{\varphi}_{111} := \widetilde{\varphi}_{11} - \widetilde{\eta}_{111} = \varphi_1 - \widetilde{\eta}_{11} - \widetilde{\eta}_{111},
$$
we have
$$
L \widetilde{\varphi}_{111} = O(e^{- \beta t}).
$$
Notice that by \eqref{eq:Lvphi1},
\begin{equation}\label{eq:Ltvphi111}
L \widetilde{\varphi}_{111} = F(\varphi) - L \widetilde{\eta}_{11} - L \widetilde{\eta}_{111} + O(e^{- \beta t}).
\end{equation}
Here we can write
\begin{equation}\label{eq:Fteta11}
F(\varphi) = \sum_{i = 2}^{[\beta]} a_i (\eta_1 + \widetilde{\eta}_{11} + \widetilde{\varphi}_{11})^i + O(e^{- \beta t}).
\end{equation}
For the lower order terms, we have
$$
\widetilde{\eta}_{11}^2 = O(e^{- 4 \rho_1 t}), ~~~~~~ \widetilde{\varphi}_{11}^2 = O(e^{- 6 \rho_1 t}), ~~~~~~ \widetilde{\eta}_{11} \widetilde{\varphi}_{11} = O(e^{- 5 \rho_1 t}),
$$
and
\begin{equation}\label{eq:eta1teta11}
\eta_1 \widetilde{\eta}_{11} = O(e^{- 3 \rho_1 t}), ~~~~~~ \eta_1 \widetilde{\varphi}_{11} = O(e^{- 4 \rho_1 t}).
\end{equation}
Recall that we now have $3 \rho_1 \leq \widetilde{\rho}_m < \beta \leq \widetilde{\rho}_{m + 1} \leq 4 \rho_1$. Therefore,
$$
\widetilde{\eta}_{11}^2 = O(e^{- \beta t}), ~~~ \widetilde{\varphi}_{11}^2 = O(e^{- \beta t}), ~~~ \widetilde{\eta}_{11} \widetilde{\varphi}_{11} = O(e^{- \beta t}) ~~~ \textmd{and} ~~~ \eta_1 \widetilde{\varphi}_{11} = O(e^{- \beta t}).
$$
This together with \eqref{eq:sumeta1}, \eqref{eq:Leta11}-\eqref{eq:eta1teta11} implies that
$$
L \widetilde{\varphi}_{111} = \sum_{i = n_1 + 1}^m \sum_{j = 0}^{ \widetilde{M}_i } a_{ij} (t) X_j (\theta) e^{ - \widetilde{\rho}_i t } + 2 a_2 \eta_1 \widetilde{\eta}_{11} - L \widetilde{\eta}_{111} + O(e^{- \beta t}).
$$
Note that we already know the expressions of $\eta_1$ and $\widetilde{\eta}_{11}$. As in Step 1, since $\mathcal{I}_1 \cap \mathcal{I}_2 = \varnothing$, by \cite[Lemma A.2 and Remark A.5]{HLL} we can take $\widetilde{\eta}_{111}$ having the form of \eqref{eq:teta111} such that
$$
\sum_{i = n_1 + 1}^m \sum_{j = 0}^{ \widetilde{M}_i } a_{ij} (t) X_j (\theta) e^{ - \widetilde{\rho}_i t } + 2 a_2 \eta_1 \widetilde{\eta}_{11} - L \widetilde{\eta}_{111} = O(e^{- \beta t}).
$$
Then $L \widetilde{\varphi}_{111} = O(e^{- \beta t})$. Since there is no $\rho_i$ between $\widetilde{\rho}_{n_1 + 1}$ and $\beta$, it follows from $\widetilde{\varphi}_{111} = O(e^{ - \widetilde{\rho}_{n_1 + 1} t })$ and Lemma \ref{lem:HanA8} that $\widetilde{\varphi}_{111} = O(e^{- \beta t})$. That is,
$$
| \varphi_1 - \widetilde{\eta}_{11} - \widetilde{\eta}_{111} | \leq C e^{- \beta t}.
$$

If $\widetilde{\rho}_m \geq 4 \rho_1$, we could proceed as at the beginning of Step 2 by taking the largest integer $n_2$ such that $\widetilde{\rho}_{n_2} < 4 \rho_1$. After finitely many steps, we obtain the desired function $\widetilde{\eta}_1$ as \eqref{eq:teta1m}, and thus
$$
| \varphi_1 - \widetilde{\eta}_1 | \leq C e^{- \beta t}.
$$

{\bf Case 5:} $\widetilde{\rho}_{l_1} < \beta < \rho_{k_1 + 1}$. From Case 4, there exists an $\widetilde{\eta}_1$ having the form of \eqref{eq:teta1m} (replacing $m$ by $l_1$) such that
$$
| \varphi_1 - \widetilde{\eta}_1 | \leq C e^{- \beta t}.
$$
Denote
$$
\widetilde{\eta}_1 (t, \theta) := \sum_{i = 1}^{l_1} \sum_{j = 0}^{ \widetilde{M}_i } \widetilde{c}_{ij} (t) X_j (\theta) e^{ - \widetilde{\rho}_i t },
$$
and
$$
\widetilde{\varphi}_1 (t, \theta) := \varphi_1 (t, \theta) - \widetilde{\eta}_1 (t, \theta).
$$
By \eqref{eq:Lvphi1} and argument in Case 4, we have
\begin{equation}\label{eq:Ltvphi1}
L \widetilde{\varphi}_1 = F(\varphi) - L \widetilde{\eta}_1 + O(e^{- \beta t})
\end{equation}
and
\begin{equation}\label{eq:F-Lteta1}
F(\varphi) - L \widetilde{\eta}_1 = O(e^{ - \widetilde{\rho}_{l_1 + 1} t }).
\end{equation}

{\bf Case 6:} $\beta = \rho_{k_1 + 1}$. We are in a similar situation as in Case 1. By \eqref{eq:Ltvphi1}, \eqref{eq:F-Lteta1} and Lemma \ref{lem:HanA8}, we obtain
$$
| \widetilde{\varphi}_1 | \leq C t e^{- \beta t}.
$$

{\bf Case 7:} $\rho_m < \beta \leq \rho_{m + 1}$ for some $m \in \{ k_1 + 1, k_1 + 2, \dots, k_2 - 1 \}$. This situation is similar to Case 2. By \eqref{eq:Ltvphi1}, \eqref{eq:F-Lteta1} and Lemma \ref{lem:HanA8}, we get
$$
\bigg| \widetilde{\varphi}_1 (t, \theta) - \sum_{i = k_1 + 1}^m c_i p_i^+ (t) X_i (\theta) e^{- \rho_i t} \bigg| \leq
\bigg\{
\aligned
& C e^{- \beta t} ~~~~~~ & \textmd{if} & ~ \rho_m < \beta < \rho_{m + 1}, \\
& C t e^{- \beta t} ~~~~~~ & \textmd{if} & ~ \beta = \rho_{m + 1},
\endaligned
$$
with some constants $c_i$ for $i = k_1 + 1, \dots, m$.

{\bf Case 8:} $\rho_{k_2} < \beta \leq \widetilde{\rho}_{l_1 + 1}$. Now we are in a similar situation as in Case 3. As above we have
$$
\bigg| \widetilde{\varphi}_1 (t, \theta) - \sum_{i = k_1 + 1}^{k_2} c_i p_i^+ (t) X_i (\theta) e^{- \rho_i t} \bigg| \leq C e^{- \beta t},
$$
with some constants $c_i$ for $i = k_1 + 1, \dots, k_2$.

Denote
$$
\eta_2 (t, \theta) := \sum_{i = k_1 + 1}^{k_2} c_i p_i^+ (t) X_i (\theta) e^{- \rho_i t},
$$
and
$$
\varphi_2 (t, \theta) := \widetilde{\varphi}_1 (t, \theta) - \eta_2 (t, \theta).
$$
Then we have $L \eta_2 = 0$, and by \eqref{eq:Ltvphi1},
\begin{equation}\label{eq:Lvphi2}
L \varphi_2 = F(\varphi) - L \widetilde{\eta}_1 + O(e^{- \beta t}).
\end{equation}

{\bf Case 9:} $\widetilde{\rho}_m < \beta \leq \widetilde{\rho}_{m + 1}$ for some $m \in \{ l_1 + 1, \dots, l_2 - 1 \}$. The argument is similar to that of Case 4. For some $\widetilde{\eta}_2$ to be determined, set $\widetilde{\varphi}_2 = \varphi_2 - \widetilde{\eta}_2$. Then by \eqref{eq:Lvphi2},
$$
L \widetilde{\varphi}_2 = F(\varphi) - L \widetilde{\eta}_1 - L \widetilde{\eta}_2 + O(e^{- \beta t}).
$$
We write
$$
F(\varphi) = \sum_{i = 2}^{[\beta]} a_i (\eta_1 + \widetilde{\eta}_1 + \eta_2 + \varphi_2)^i + O(e^{- \beta t}).
$$
Recall that, in Case 4 and Case 5 we use $L \widetilde{\eta}_1$ to cancel the terms with decay $e^{ - \widetilde{\rho}_i t }$ in the expansion of $F(\varphi)$ for $i = 1, \dots, l_1$. Proceeding similarly, since $\mathcal{I}_1 \cap \mathcal{I}_2 = \varnothing$, we can find an $\widetilde{\eta}_2$ having the form of
$$
\widetilde{\eta}_2 (t, \theta) = \sum_{i = l_1 + 1}^m \sum_{j = 0}^{ \widetilde{M}_i } \widetilde{c}_{ij} (t) X_j (\theta) e^{ - \widetilde{\rho}_i t }
$$
to cancel the terms with decay $e^{ - \widetilde{\rho}_i t }$ in the expansion of $F(\varphi)$ for $i = l_1 + 1, \dots, m$, where $\widetilde{c}_{ij}$ is a smooth periodic function with the same period as $\xi$. Then we conclude that $L \widetilde{\varphi}_2 = O(e^{- \beta t})$. Since there is no $\rho_i$ between $\widetilde{\rho}_{l_1 + 1}$ and $\beta$, it follows from $\widetilde{\varphi}_2 = O(e^{ - \widetilde{\rho}_{l_1 + 1} t })$ and Lemma \ref{lem:HanA8} that
$$
| \varphi_2 - \widetilde{\eta}_2 | = | \widetilde{\varphi}_2 | \leq C e^{- \beta t}.
$$

Since $\beta$ is finite, repeating the similar arguments as above, we obtain \eqref{eq:csc-expans} for the case $\mathcal{I}_1 \cap \mathcal{I}_2 = \varnothing$ by denoting the index set $\mathcal{I} [\xi] = \{ \mu_i \}_{i \geq 1}$, which is a strictly increasing sequence of positive constants.

Next we discuss the general case, i.e., $\mathcal{I}_1 \cap \mathcal{I}_2 \neq \varnothing$. For an illustration, we consider $\rho_{k_1} = \widetilde{\rho}_1$ instead of the strict inequality, which is the smallest element in $\mathcal{I}_1 \cap \mathcal{I}_2$. Let $k_* \in \{ 1, \dots, k_1 - 1 \}$ such that
$$
\rho_{k_*} < \rho_{k_* + 1} = \cdots = \rho_{k_1} = \widetilde{\rho}_1.
$$
We will modify the previous discussion to deal with this situation. If $\beta < \rho_{k_* + 1}$, the discussion is the same as in Case 1 and Case 2. If $\beta = \rho_{k_* + 1} = \widetilde{\rho}_1$, we still have
$$
| \varphi - \eta_1 | \leq C t e^{- \beta t},
$$
where
\begin{equation}\label{eq:eta1mod}
\eta_1 (t, \theta) = \sum_{i = 1}^{k_*} c_i p_i^+ (t) X_i (\theta) e^{- \rho_i t}.
\end{equation}
Suppose $\widetilde{\rho}_m < \beta \leq \widetilde{\rho}_{m + 1} \leq 3 \rho_1$ for some $m \in \{ 1, \dots, l_1 - 1 \}$. Let $\eta_1$ be defined as in \eqref{eq:eta1mod} and $\varphi_1 = \varphi - \eta_1$. Take $\widetilde{\eta}_1$ to be determined later and define $\widetilde{\varphi}_1 = \varphi_1 - \widetilde{\eta}_1$. By \eqref{eq:Ltvphi1m}-\eqref{eq:sumeta1} we have
$$
L \widetilde{\varphi}_1 = - L \widetilde{\eta}_1 + \sum_{i = 1}^m \sum_{j = 0}^{ \widetilde{M}_i } a_{ij} (t) X_j (\theta) e^{ - \widetilde{\rho}_i t } + O(e^{- \beta t}).
$$
When $\widetilde{\rho}_i = \rho_j \in \mathcal{I}_1 \cap \mathcal{I}_2$ for $1 \leq i \leq m$ and $0 \leq j \leq \widetilde{M}_i$, by \cite[Lemma A.2 and Remark A.5]{HLL}, instead of \eqref{eq:Ljcijrhoi} there exist two smooth periodic functions $\widetilde{b}_{ij}$ and $\widetilde{d}_{ij}$ such that
$$
L_j ( \widetilde{b}_{ij} (t) e^{ - \widetilde{\rho}_i t } + \widetilde{d}_{ij} (t) t e^{ - \widetilde{\rho}_i t } ) = a_{ij} (t) e^{ - \widetilde{\rho}_i t }.
$$
Hence, we replace those terms $\widetilde{c}_{ij} (t)$ in the summation \eqref{eq:teta1m} by $\widetilde{b}_{ij} (t) + \widetilde{d}_{ij} (t) t$ to define the new $\widetilde{\eta}_1$ when $\widetilde{\rho}_i = \rho_j \in \mathcal{I}_1 \cap \mathcal{I}_2$ for $1 \leq i \leq m$ and $0 \leq j \leq \widetilde{M}_i$. Then we get $L \widetilde{\varphi}_1 = O(e^{- \beta t})$, and thus
$$
| \varphi_1 - \widetilde{\eta}_1 | = | \widetilde{\varphi}_1 | \leq
\bigg\{
\aligned
& C t e^{- \beta t} ~~~~~~ && \textmd{if} ~ \beta = \widetilde{\rho}_{m + 1} ~ \textmd{and} ~ \widetilde{\rho}_{m + 1} \in \mathcal{I}_1 \cap \mathcal{I}_2, \\
& C e^{- \beta t} ~~~~~~ && \textmd{otherwise}.
\endaligned
$$

In the rest of the proof, when $\widetilde{\rho}_i = \rho_j \in \mathcal{I}_1 \cap \mathcal{I}_2$, an extra power of $t$ appears when solving $L_j ( \widetilde{c}_{ij} (t) e^{ - \widetilde{\rho}_i t } ) = a_{ij} (t) e^{ - \widetilde{\rho}_i t }$ according to \cite[Lemma A.2 and Remark A.5]{HLL}. Such a power of $t$ will generate more powers of $t$ upon iteration. The worst situation is that $\widetilde{\rho}_i \in \mathcal{I}_1 \cap \mathcal{I}_2$ for every $i \geq 1$. Therefore, when $\beta = \mu_{m + 1}$, we have at most $m$-th power of $t$ on the right hand side of \eqref{eq:csc-expans}. The proof of Theorem \ref{thm:csc-expans} is completed.
\end{proof}

Before proving Theorem \ref{thm:csc-expans-2}, we introduce a family of solutions of \eqref{eq:v} with $K(x) \equiv K(0)$. Let $\xi$ be a positive periodic solution of \eqref{eq:xi}. Then $u_0 (|x|) = |x|^{ - \frac{n - 2}{2} } \xi (- \ln |x|)$ solves \eqref{eq:u0}. For $a \in \R^n$, by the conformal invariance we know that
$$
\aligned
u_a (x) & = \bigg( \frac{1}{|x|} \bigg)^{n - 2} \bigg| \frac{x}{|x|^2} - a \bigg|^{ - \frac{n - 2}{2} } \xi \bigg( \ln \bigg| \frac{x}{|x|^2} - a \bigg| \bigg) \\
& = |x|^{ - \frac{n - 2}{2} } | \theta - a |x| |^{ - \frac{n - 2}{2} } \xi ( - \ln |x| + \ln | \theta - a |x| | )
\endaligned
$$
still solves the equation in \eqref{eq:u0}, where $\theta = x/|x|$. Define the associated function
$$
\xi_a (t, \theta) = | \theta - e^{- t} a |^{ - \frac{n - 2}{2} } \xi ( t + \ln | \theta - e^{- t} a | ) ~~~~~~ \textmd{for} ~ (t, \theta) \in (\ln |a|, \infty) \times \Sp^{n - 1}.
$$
Then $\xi_a$ satisfies \eqref{eq:v} with $K(x) \equiv K(0)$. We would expand $\xi_a$ to the second order. Note that
$$
\aligned
| \theta - e^{- t} a |^{ - \frac{n - 2}{2} } = &~ 1 + \frac{n - 2}{2} e^{- t} \langle a, \theta \rangle + \frac{(n + 2) (n - 2)}{8} e^{- 2 t} \langle a, \theta \rangle^2 \\
&~ - \frac{n - 2}{4} e^{- 2 t} |a|^2 + O(e^{- 3 t}),
\endaligned
$$
where $\langle a, \theta \rangle := a_1 \theta_1 + \cdots + a_n \theta_n$. Similarly we have
$$
\ln | \theta - e^{- t} a | = - e^{- t} \langle a, \theta \rangle - e^{- 2 t} \langle a, \theta \rangle^2 + \frac{1}{2} e^{- 2 t} |a|^2 + O(e^{- 3 t}),
$$
and so
$$
\aligned
\xi ( t + \ln | \theta - e^{- t} a | ) = &~ \xi (t) + \xi' (t) \bigg( - e^{- t} \langle a, \theta \rangle - e^{- 2 t} \langle a, \theta \rangle^2 + \frac{1}{2} e^{- 2 t} |a|^2 \bigg) \\
&~ + \frac{1}{2} \xi'' (t) e^{- 2 t} \langle a, \theta \rangle^2 + O(e^{- 3 t}) \\
= &~ \xi (t) + \xi' (t) \bigg( - e^{- t} \langle a, \theta \rangle - e^{- 2 t} \langle a, \theta \rangle^2 + \frac{1}{2} e^{- 2 t} |a|^2 \bigg) \\
&~ + \frac{1}{2} \bigg( \frac{(n - 2)^2}{4} \xi (t) - K(0) \xi (t)^\frac{n + 2}{n - 2} \bigg) e^{- 2 t} \langle a, \theta \rangle^2 + O(e^{- 3 t}),
\endaligned
$$
where we used the fact that $\xi$ is a solution of \eqref{eq:xi}. Putting these altogether we obtain
\begin{equation}\label{eq:xia}
\aligned
\xi_a (t, \theta) = &~ \xi (t) + e^{- t} \bigg( \frac{n - 2}{2} \xi (t) - \xi' (t) \bigg) \langle a, \theta \rangle \\
&~ + e^{- 2 t} \bigg( \frac{n (n - 2)}{4} \xi (t) - \frac{n}{2} \xi' (t) - \frac{1}{2} K(0) \xi (t)^\frac{n + 2}{n - 2} \bigg) \langle a, \theta \rangle^2 \\
&~ + e^{- 2 t} \bigg( - \frac{n - 2}{4} \xi (t) + \frac{1}{2} \xi' (t) \bigg) |a|^2 + O(e^{- 3 t}) \\
= &~ \xi (t) + e^{- t} \bigg( \frac{n - 2}{2} \xi (t) - \xi' (t) \bigg) \langle a, \theta \rangle \\
&~ + e^{- 2 t} \bigg( - \frac{1}{2} K(0) \xi (t)^\frac{n + 2}{n - 2} \bigg) \langle a, \theta \rangle^2 \\
&~ + e^{- 2 t} \bigg( - \frac{n - 2}{8} \xi (t) + \frac{1}{4} \xi' (t) \bigg) \Delta_\theta (\langle a, \theta \rangle^2) + O(e^{- 3 t}).
\endaligned
\end{equation}
Here we used the fact $2 |a|^2 = 2 n \langle a, \theta \rangle^2 + \Delta_\theta (\langle a, \theta \rangle^2)$ in the last equality. 
Now we show the exact formula for the term with decay $e^{- 2 t}$ in the asymptotic expansion \eqref{eq:csc-expans} when $n \geq 6$.

\begin{proof}[Proof of Theorem \ref{thm:csc-expans-2}] Since $n \geq 6$, by Lemma \ref{lem:rhoi2} we know $\rho_{n + 1} > 2$. As in Case 4 in the proof of Theorem \ref{thm:csc-expans}, let $\varphi = v - \xi$ and
$$
\xi_1 (t, \theta) = e^{- t} \bigg( \frac{n - 2}{2} \xi (t) - \xi' (t) \bigg) Y(\theta),
$$
where $Y = Y(\theta)$ is a spherical harmonic of degree $1$ such that $\varphi - \xi_1 = O(e^{- 2 t})$. Denote $\varphi_1 = \varphi - \xi_1$. Take some function $\xi_2$ to be determined later, and set $\widetilde{\varphi}_1 = \varphi_1 - \xi_2$. Then
$$
\aligned
L \widetilde{\varphi}_1 & = K (\xi + \varphi)^\frac{n + 2}{n - 2} - K(0) \xi^\frac{n + 2}{n - 2} - \frac{n + 2}{n - 2} K(0) \xi^\frac{4}{n - 2} \varphi - L \xi_2 \\
& = \frac{2 (n + 2)}{(n - 2)^2} K(0) \xi^\frac{6 - n}{n - 2} \varphi^2 - L \xi_2 + O(e^{- \gamma t}) \\
& = \frac{2 (n + 2)}{(n - 2)^2} K(0) \xi^\frac{6 - n}{n - 2}\xi_1^2 - L \xi_2 + O(e^{- \gamma t}) \\
& = \frac{2 (n + 2)}{(n - 2)^2} K(0) \xi^\frac{6 - n}{n - 2} \bigg( \frac{n - 2}{2} \xi - \xi' \bigg)^2 Y^2 e^{- 2 t} - L \xi_2 + O(e^{- \gamma t}),
\endaligned
$$
where $2 < \gamma < \min \{ 3, \rho_{n + 1}, \beta \}$. By the asymptotic expansion of $\xi_a$ in \eqref{eq:xia}, we take
$$
\xi_2 (t, \theta) = e^{- 2 t} \bigg[ - \frac{1}{2} K(0) \xi (t)^\frac{n + 2}{n - 2} Y(\theta)^2 + \bigg( - \frac{n - 2}{8} \xi (t) + \frac{1}{4} \xi' (t) \bigg) \Delta_\theta (Y^2) \bigg].
$$
One can check that $\xi_2$ satisfies
$$
L \xi_2 = \frac{2 (n + 2)}{(n - 2)^2} K(0) \xi^\frac{6 - n}{n - 2} \bigg( \frac{n - 2}{2} \xi - \xi' \bigg)^2 Y^2 e^{- 2 t}.
$$
Therefore, $L \widetilde{\varphi}_1 = O(e^{- \gamma t})$. Note that $\widetilde{\varphi}_1 = O(e^{- 2 t})$ and there is no $\rho_i$ between $2$ and $\gamma$. It follows from Lemma \ref{lem:HanA8} that we have $\widetilde{\varphi}_1 = O(e^{- \gamma t})$. The proof of Theorem \ref{thm:csc-expans-2} is completed.
\end{proof}

\begin{proof}[Proof of Theorem \ref{thm:csc}] This is a direct consequence of Theorem \ref{thm:csc-expans}, \cite[Theorem 1.2]{CL99a} and \cite[Theorem 2]{TZ06}.
\end{proof}

\subsection{Existence}

Before our proof, we introduce a weighted H{\" o}lder space in $[t_0, \infty) \times \Sp^{n - 1}$.

\begin{definition} For $t_0 > 0$, $k \in \N$, $\alpha \in (0, 1)$ and $\gamma \in \R$, set
$$
\| f \|_{ C_\gamma^k ([t_0, \infty) \times \Sp^{n - 1}) } = \sum_{j = 0}^k \sup_{ (t, \theta) \in [t_0, \infty) \times \Sp^{n - 1} } e^{\gamma t} | \nabla^j f(t, \theta) |,
$$
and
$$
\| f \|_{ C_\gamma^{k, \alpha} ([t_0, \infty) \times \Sp^{n - 1}) } = \| f \|_{ C_\gamma^k ([t_0, \infty) \times \Sp^{n - 1}) } + \sup_{ t \geq t_0 + 1 } e^{\gamma t} [ \nabla^k f ]_{ C^\alpha ([t - 1, t + 1] \times \Sp^{n - 1}) },
$$
where $[\cdot]_{C^\alpha}$ is the usual H{\" o}lder semi-norm.
\end{definition}

\begin{proof}[Proof of Theorem \ref{thm:csc-existv1}] Without loss of generality, we assume that $K(0) = 1$. Since $| \nabla K(x) | = O(|x|^{\beta - 1})$, we have $K(x) = 1 + O(|x|^\beta)$ near the origin. This implies that
$$
K(e^{- t} \theta) - 1 \in C_\beta^1 ([1, \infty) \times \Sp^{n - 1}).
$$
Suppose that $\widehat{v} \in C^{2, \alpha} (\R_+ \times \Sp^{n - 1})$ satisfies \eqref{eq:hatv-xi} and \eqref{eq:Mhatv}. Our target is to find $\varphi \in C_\beta^{2, \alpha} ([t_0, \infty) \times \Sp^{n - 1})$ such that
$$
\mathcal{M} (\widehat{v} + \varphi) = 0,
$$
where the operator $\mathcal{M}$ is defined as in \eqref{eq:csc-M}. It is equivalent to
\begin{equation}\label{24=Lpoi}
L \varphi = - \mathcal{M} (\widehat{v}) + P(\varphi),
\end{equation}
where the operator $P$ is given by
$$
P(\varphi) := K (\widehat{v} + \varphi)^\frac{n + 2}{n - 2} - K \widehat{v}^\frac{n + 2}{n - 2} - \frac{n + 2}{n - 2} \xi^\frac{4}{n - 2} \varphi.
$$
With the operator $L^{- 1}$ introduced in \cite[Remark 2.10]{HL}, we can rewrite \eqref{24=Lpoi} further as
$$
\varphi = L^{- 1} [ - \mathcal{M} (\widehat{v}) + P(\varphi) ].
$$
Define the mapping $T$ as
$$
T(\varphi) = L^{- 1} [ - \mathcal{M} (\widehat{v}) + P(\varphi) ].
$$
We claim that $T$ is a contraction on a closed ball in $C_\beta^{2, \alpha} ([t_0, \infty) \times \Sp^{n - 1})$ for some $t_0$ large. Set
$$
\mathcal{B}_{t_0, R} := \Big\{ \varphi \in C_\beta^{2, \alpha} ([t_0, \infty) \times \Sp^{n - 1}) : \| \varphi \|_{ C_\beta^{2, \alpha} ([t_0, \infty) \times \Sp^{n - 1}) } \leq R \Big\}.
$$

First, we prove that $T$ maps $\mathcal{B}_{t_0, R}$ to itself for some fixed $R > 0$ and any $t_0$ sufficiently large. By \eqref{eq:hatv-xi}, there exists $t_1 \geq 1$ such that for any $(t, \theta) \in [t_1, \infty) \times \Sp^{n - 1}$,
$$
\frac{1}{2} \min\nolimits_\R \xi \leq \widehat{v} (t, \theta) \leq \frac{3}{2} \max\nolimits_\R \xi.
$$
Let $t_0 \geq t_R + 1$ with
$$
t_R := \max \bigg\{ t_1, \frac{1}{\beta} \ln \bigg( \frac{4 R}{ \min_\R \xi } \bigg) \bigg\}.
$$
By the definition of $t_R$, one can easily check that $T$ is well-defined on $\mathcal{B}_{t_0, R}$. Now we estimate $- \mathcal{M} (\widehat{v}) + P(\varphi)$. By \eqref{eq:Mhatv} we have
$$
\| \mathcal{M} (\widehat{v}) \|_{ C_\beta^1 ([t_0, \infty) \times \Sp^{n - 1}) } \leq C_1.
$$
For any $\varphi \in \mathcal{B}_{t_0, R}$ for some $R$ to be determined, set
\begin{equation}\label{eq:Qvarphi-int}
Q(\varphi) = \frac{n + 2}{n - 2} \int_0^1 \Big[ (\widehat{v} + s \varphi)^\frac{4}{n - 2} - \xi^\frac{4}{n - 2} \Big] {\rm d} s.
\end{equation}
Then
$$
P(\varphi) = \varphi Q(\varphi) + (K - 1) \Big[ (\widehat{v} + \varphi)^\frac{n + 2}{n - 2} - \widehat{v}^\frac{n + 2}{n - 2} \Big].
$$
Note that by \eqref{eq:hatv-xi},
$$
| (\widehat{v} - \xi) | + | \nabla (\widehat{v} - \xi) | \leq \epsilon (t),
$$
where $\epsilon (t)$ is a decreasing function with $\epsilon (t) \to 0$ as $t \to \infty$, and
$$
| \varphi | + | \nabla \varphi | \leq R e^{- \beta t}.
$$
Then for $t \geq t_0$,
\begin{equation}\label{eq:Qvarphi-esti}
| Q(\varphi) | + | \nabla Q(\varphi) | \leq C_2 (\epsilon (t) + R e^{- \beta t}).
\end{equation}
Hence
$$
\| \varphi Q(\varphi) \|_{ C_\beta^1 ([t_0, \infty) \times \Sp^{n - 1}) } \leq C_2 (\epsilon (t_0) + R e^{- \beta t_0}) R.
$$
Since
$$
(\widehat{v} + \varphi)^\frac{n + 2}{n - 2} - \widehat{v}^\frac{n + 2}{n - 2} = \varphi \cdot \frac{n + 2}{n - 2} \int_0^1 (\widehat{v} + s \varphi)^\frac{4}{n - 2} {\rm d} s
$$
and $K(e^{- t} \theta) - 1 \in C_\beta^1 ([1, \infty) \times \Sp^{n - 1})$, we obtain
$$
\bigg\| (K - 1) \Big[ (\widehat{v} + \varphi)^\frac{n + 2}{n - 2} - \widehat{v}^\frac{n + 2}{n - 2} \Big] \bigg\|_{ C_\beta^1 ([t_0, \infty) \times \Sp^{n - 1}) } \leq C_3 R e^{- \beta t_0}.
$$
Eventually, we have
$$
\| P(\varphi) \|_{ C_\beta^1 ([t_0, \infty) \times \Sp^{n - 1}) } \leq C_2 (\epsilon (t_0) + R e^{- \beta t_0}) R + C_3 R e^{- \beta t_0}.
$$
By \cite[Theorem 2.9]{HL}, we get
$$
\| T(\varphi) \|_{ C_\beta^{2, \alpha} ([t_0, \infty) \times \Sp^{n - 1}) } \leq C [ C_1 + ( C_2 \epsilon (t_0) + C_2 R e^{- \beta t_0} + C_3 e^{- \beta t_0} ) R ],
$$
where $C$, $C_1$, $C_2$ and $C_3$ are positive constants independent of $t_0$ and $R$. Take $R > 2 C C_1$ and then take $t_0$ large such that $C ( C_2 \epsilon (t_0) + C_2 R e^{- \beta t_0} + C_3 e^{- \beta t_0} ) \leq 1/2$. Then
$$
\| T(\varphi) \|_{ C_\beta^{2, \alpha} ([t_0, \infty) \times \Sp^{n - 1}) } \leq R.
$$
This shows that $T$ maps $\mathcal{B}_{t_0, R}$ to itself.

Second, we verify that $T$ is a contraction mapping for $t_0$ sufficiently large. For any $\varphi_1, \varphi_2 \in \mathcal{B}_{t_0, R}$, we have
$$
T(\varphi_1) - T(\varphi_2) = L^{- 1} [ P(\varphi_1) - P(\varphi_2) ]
$$
and
$$
P(\varphi_1) - P(\varphi_2) = \varphi_1 Q(\varphi_1) - \varphi_2 Q(\varphi_2) + (K - 1) \Big[ (\widehat{v} + \varphi_1)^\frac{n + 2}{n - 2} - (\widehat{v} + \varphi_2)^\frac{n + 2}{n - 2} \Big].
$$
Notice that
$$
\varphi_1 Q(\varphi_1) - \varphi_2 Q(\varphi_2) = (\varphi_1 - \varphi_2) Q(\varphi_1) + \varphi_2 ( Q(\varphi_1) - Q(\varphi_2) ).
$$
By \eqref{eq:Qvarphi-int}, we have
$$
Q(\varphi_1) - Q(\varphi_2) = \frac{n + 2}{n - 2} \int_0^1 \Big[ (\widehat{v} + s \varphi_1)^\frac{4}{n - 2} - (\widehat{v} + s \varphi_2)^\frac{4}{n - 2} \Big] {\rm d} s.
$$
Then
$$
| Q(\varphi_1) - Q(\varphi_2) | + | \nabla ( Q(\varphi_1) - Q(\varphi_2) ) | \leq C ( | \varphi_1 - \varphi_2 | + | \nabla ( \varphi_1 - \varphi_2 ) | ).
$$
By \eqref{eq:Qvarphi-esti}, we obtain that for any $t \geq t_0$,
$$
| \varphi_1 Q(\varphi_1) - \varphi_2 Q(\varphi_2) | + | \nabla ( \varphi_1 Q(\varphi_1) - \varphi_2 Q(\varphi_2) ) | \leq C (\epsilon (t) + R e^{- \beta t}) ( | \varphi_1 - \varphi_2 | + | \nabla ( \varphi_1 - \varphi_2 ) | ),
$$
and hence
$$
\| \varphi_1 Q(\varphi_1) - \varphi_2 Q(\varphi_2) \|_{ C_\beta^1 ([t_0, \infty) \times \Sp^{n - 1}) } \leq C (\epsilon (t_0) + R e^{- \beta t_0}) \| \varphi_1 - \varphi_2 \|_{ C_\beta^1 ([t_0, \infty) \times \Sp^{n - 1}) }.
$$
We also have
$$
\bigg\| (K - 1) \Big[ (\widehat{v} + \varphi_1)^\frac{n + 2}{n - 2} - (\widehat{v} + \varphi_2)^\frac{n + 2}{n - 2} \Big] \bigg\|_{ C_\beta^1 ([t_0, \infty) \times \Sp^{n - 1}) } \leq C e^{- \beta t_0} \| \varphi_1 - \varphi_2 \|_{ C_\beta^1 ([t_0, \infty) \times \Sp^{n - 1}) }.
$$
Therefore, from \cite[Theorem 2.9]{HL} we get
$$
\aligned
\| T(\varphi_1) - T(\varphi_2) \|_{ C_\beta^{2, \alpha} ([t_0, \infty) \times \Sp^{n - 1}) } & \leq C (\epsilon (t_0) + R e^{- \beta t_0} + e^{- \beta t_0}) \| \varphi_1 - \varphi_2 \|_{ C_\beta^1 ([t_0, \infty) \times \Sp^{n - 1}) } \\
& \leq C (\epsilon (t_0) + R e^{- \beta t_0} + e^{- \beta t_0}) \| \varphi_1 - \varphi_2 \|_{ C_\beta^{2, \alpha} ([t_0, \infty) \times \Sp^{n - 1}) }.
\endaligned
$$
Thus, $T$ is a contraction mapping when $t_0$ is sufficiently large.

By the contraction mapping principle, there exists $\varphi \in C_\beta^{2, \alpha} ([t_0, \infty) \times \Sp^{n - 1})$ such that $T \varphi = \varphi$. Hence $v = \widehat{v} + \varphi$ is a solution of \eqref{eq:v} in $(t_0, \infty) \times \Sp^{n - 1}$. The proof of Theorem \ref{thm:csc-existv1} is finished.
\end{proof}

\begin{proof}[Proof of Theorem \ref{thm:csc-existv2}] Without loss of generality, we assume that $K(0) = 1$. Since $| \nabla K(x) | = O(|x|^{\beta - 1})$, we have $K(x) = 1 + O(|x|^\beta)$ near the origin. This implies that $K(e^{- t} \theta) - 1 \in C_\beta^1 ([1, \infty) \times \Sp^{n - 1})$. Our aim is to construct $\varphi \in C_\beta^{2, \alpha} ([t_0, \infty) \times \Sp^{n - 1})$ such that
$$
\mathcal{M} (\xi + \varphi) = 0,
$$
where the operator $\mathcal{M}$ is defined as in \eqref{eq:csc-M}. This is equivalent to solving
\begin{equation}\label{24=09nh7}
L \varphi = (K - 1) (\xi + \varphi)^\frac{n + 2}{n - 2} + F(\varphi),
\end{equation}
where the operator $F$ is given by
$$
F(\varphi) := (\xi + \varphi)^\frac{n + 2}{n - 2} - \xi^\frac{n + 2}{n - 2} - \frac{n + 2}{n - 2} \xi^\frac{4}{n - 2} \varphi.
$$

We first consider the case that $\beta > 1$ and $\beta \notin \{ \rho_i \}_{i \geq 1}$. In this case, with the operator $L^{- 1}$ introduced in \cite[Remark 2.10]{HL}, we can rewrite \eqref{24=09nh7} as
$$
\varphi = L^{- 1} \Big[ (K - 1) (\xi + \varphi)^\frac{n + 2}{n - 2} + F(\varphi) \Big].
$$
Define the mapping $T$ as
$$
T(\varphi) = L^{- 1} \Big[ (K - 1) (\xi + \varphi)^\frac{n + 2}{n - 2} + F(\varphi) \Big].
$$
We claim that $T$ is a contraction on some closed ball in $C_\beta^{2, \alpha} ([t_0, \infty) \times \Sp^{n - 1})$ for some $t_0$ large. Set
$$
\mathcal{B}_{t_0, R} := \Big\{ \varphi \in C_\beta^{2, \alpha} ([t_0, \infty) \times \Sp^{n - 1}) : \| \varphi \|_{ C_\beta^{2, \alpha} ([t_0, \infty) \times \Sp^{n - 1}) } \leq R \Big\}.
$$
We now show that $T$ maps $\mathcal{B}_{t_0, R}$ to itself for some fixed $R > 0$ and any $t_0$ sufficiently large. Let $t_0 \geq t_R + 1$ with
$$
t_R := \max \bigg\{ 1, \frac{1}{\beta} \ln \bigg( \frac{2 R}{ \min_\R \xi } \bigg) \bigg\}.
$$
By the definition of $t_R$, one can easily check that $T$ is well-defined on $\mathcal{B}_{t_0, R}$. We estimate $(K - 1) (\xi + \varphi)^\frac{n + 2}{n - 2} + F(\varphi)$. Since $K(e^{- t} \theta) - 1 \in C_\beta^1 ([1, \infty) \times \Sp^{n - 1})$, we have that for any $\varphi \in \mathcal{B}_{t_0, R}$ with $R$ to be determined,
$$
\Big\| (K - 1) (\xi + \varphi)^\frac{n + 2}{n - 2} \Big\|_{ C_\beta^1 ([t_0, \infty) \times \Sp^{n - 1}) } \leq C_1.
$$
Let
$$
Q(\varphi) := \frac{n + 2}{n - 2} \int_0^1 \Big[ (\xi + s \varphi)^\frac{4}{n - 2} - \xi^\frac{4}{n - 2} \Big] {\rm d} s.
$$
Then $F(\varphi) = \varphi Q(\varphi)$. Note that
$$
| \varphi | + | \nabla \varphi | \leq R e^{- \beta t}.
$$
Thus, for any $t \geq t_0$,
$$
| Q(\varphi) | + | \nabla Q(\varphi) | \leq C_2 R e^{- \beta t},
$$
and hence
$$
\| F (\varphi) \|_{ C_\beta^1 ([t_0, \infty) \times \Sp^{n - 1}) } = \| \varphi Q(\varphi) \|_{ C_\beta^1 ([t_0, \infty) \times \Sp^{n - 1}) } \leq C_2 R^2 e^{- \beta t_0}.
$$
By \cite[Theorem 2.9]{HL}, we get
$$
\| T(\varphi) \|_{ C_\beta^{2, \alpha} ([t_0, \infty) \times \Sp^{n - 1}) } \leq C (C_1 + C_2 R^2 e^{- \beta t_0}),
$$
where $C$, $C_1$ and $C_2$ are positive constants independent of $t_0$ and $R$. Take $R > 2 C C_1$ and then take $t_0$ large such that $C C_2 R e^{- \beta t_0} \leq 1/2$. Then
$$
\| T(\varphi) \|_{ C_\beta^{2, \alpha} ([t_0, \infty) \times \Sp^{n - 1}) } \leq R.
$$
That is, $T(\mathcal{B}_{t_0, R}) \subset \mathcal{B}_{t_0, R}$. Next we verify that $T$ is a contraction mapping for $t_0$ sufficiently large. For any $\varphi_1, \varphi_2 \in \mathcal{B}_{t_0, R}$, we have
$$
T(\varphi_1) - T(\varphi_2) = L^{- 1} \Big[ (K - 1) \Big( (\xi + \varphi_1)^\frac{n + 2}{n - 2} - (\xi + \varphi_1)^\frac{n + 2}{n - 2} \Big) + F(\varphi_1) - F(\varphi_2) \Big].
$$
Similar to the proof of Theorem \ref{thm:csc-existv1}, we have
$$
\| F(\varphi_1) - F(\varphi_2) \|_{ C_\beta^1 ([t_0, \infty) \times \Sp^{n - 1}) } \leq C R e^{- \beta t_0} \| \varphi_1 - \varphi_2 \|_{ C_\beta^1 ([t_0, \infty) \times \Sp^{n - 1}) }
$$
and
$$
\Big\| (\xi + \varphi_1)^\frac{n + 2}{n - 2} - (\xi + \varphi_2)^\frac{n + 2}{n - 2} \Big\|_{ C_\beta^1 ([t_0, \infty) \times \Sp^{n - 1}) } \leq C \| \varphi_1 - \varphi_2 \|_{ C_\beta^1 ([t_0, \infty) \times \Sp^{n - 1}) }.
$$
Therefore, using \cite[Theorem 2.9]{HL} we get
$$
\aligned
\| T(\varphi_1) - T(\varphi_2) \|_{ C_\beta^{2, \alpha} ([t_0, \infty) \times \Sp^{n - 1}) } & \leq C (R e^{- \beta t_0} + e^{- \beta t_0}) \| \varphi_1 - \varphi_2 \|_{ C_\beta^1 ([t_0, \infty) \times \Sp^{n - 1}) } \\
& \leq C (R e^{- \beta t_0} + e^{- \beta t_0}) \| \varphi_1 - \varphi_2 \|_{ C_\beta^{2, \alpha} ([t_0, \infty) \times \Sp^{n - 1}) }.
\endaligned
$$
Hence, $T$ is a contraction when $t_0$ is sufficiently large. By the contraction mapping principle, there exists $\varphi \in C_\beta^{2, \alpha} ([t_0, \infty) \times \Sp^{n - 1})$ such that $T \varphi = \varphi$. Thus, $v = \xi + \varphi$ is a solution of \eqref{eq:v} in $(t_0, \infty) \times \Sp^{n - 1}$.

Finally, we consider the case $\beta > 1$ and $\beta \in \{ \rho_i \}_{i \geq 1} =: \mathcal{I}_1$. Let $\beta' \in (\beta/2, \beta)$ such that $\mathcal{I}_1 \cap [\beta', \beta) = \varnothing$. Then $\beta' > 1$, $\beta' \notin \mathcal{I}_1$ and $| \nabla K(x) | = O(|x|^{\beta' - 1})$. From the above arguments, we know that there exists $\varphi \in C_{\beta'}^{2, \alpha} ([t_0, \infty) \times \Sp^{n - 1})$ satisfying $T \varphi = \varphi$. Therefore,
$$
L \varphi = F(\varphi) + (K - 1) (\xi + \varphi)^\frac{n + 2}{n - 2} = O(e^{- 2 \beta' t}) + O(e^{- \beta t}) = O(e^{- \beta t}).
$$
By Lemma \ref{lem:HanA8}, we obtain $| \varphi | \leq C t e^{- \beta t}$. Moreover, $v = \xi + \varphi$ is a solution of \eqref{eq:v} in $(t_0, \infty) \times \Sp^{n - 1}$. The proof of Theorem \ref{thm:csc-existv2} is completed.
\end{proof}

\section{Anisotropic elliptic equation}\label{Sec=04}

In this section, we first prove the arbitrary-order expansion in Theorem \ref{thm:ckn-expans}, where the various properties established in Appendix \ref{AppendixB} will be used. Then we show the existence result in Theorem \ref{thm:ckn-existw1} based on the contraction mapping principle.

\subsection{Asymptotic expansion}

\begin{proof}[Proof of Theorem \ref{thm:ckn-expans}] Throughout the proof, we always assume that $\zeta$ is a positive nonconstant periodic solution of \eqref{eq:zeta}. The proof is similar for the case when $\zeta$ is the positive constant solution. Define
$$
\phi (t, \theta) = w(t, \theta) - \zeta (t) ~~~~~~ \textmd{for} ~ (t, \theta) \in \R_+ \times \Sp^{n - 1}.
$$
Then by \eqref{eq:w2zeta} we have
$$
\phi (t, \theta) \to 0 ~~~~~~ \textmd{as} ~ t \to \infty ~ \textmd{uniformly for} ~ \theta \in \Sp^{n - 1}.
$$
Moreover, $\phi$ solves
\begin{equation}\label{eq:Lphi}
\mathcal{L} \phi = (\zeta + \phi)^{p - 1} - \zeta^{p - 1} - (p - 1) \zeta^{p - 2} \phi =: \mathcal{F} (\phi).
\end{equation}
We decompose the index set $\mathcal{S} [\zeta]$ into two parts. Let
$$
\mathcal{S}_1 = \{ \sigma_i : i \geq 1 \}
$$
and
$$
\mathcal{S}_2 = \bigg\{ \sum_{i = 1}^k n_i \sigma_i > 0 : n_i \in \N ~ \textmd{and} ~ \sum_{i = 1}^k n_i \geq 2 \bigg\}.
$$
Denote the set $\mathcal{S}_2$ by a strictly increasing sequence $\{ \widetilde{\sigma}_i \}_{i \geq 1}$. Note that $\widetilde{\sigma}_1 = 2 \sigma_1$. First, we prove \eqref{eq:ckn-expans} for the simple case: $\mathcal{S}_1 \cap \mathcal{S}_2 = \varnothing$. In this case, one can arrange $\mathcal{S} [\zeta]$ as follows:
$$
\aligned
0 < \sigma_1 \leq \cdots \leq \sigma_{k_1} < \widetilde{\sigma}_1 < \cdots & < \widetilde{\sigma}_{l_1} \\
& < \sigma_{k_1 + 1} \leq \cdots \leq \sigma_{k_2} < \widetilde{\sigma}_{l_1 + 1} < \cdots.
\endaligned
$$

{\bf Step 1.} By \cite[Theorem 1.2]{HLW}, we know that $\phi = O(e^{- \varepsilon t})$ for some $\varepsilon > 0$. Without loss of generality, we can assume that $\varepsilon \neq 2^{- k} \sigma_1$ for any positive integer $k$. Since $| \mathcal{F} (\phi) | \leq C \phi^2$, we have $\mathcal{L} \phi = O(e^{- 2 \varepsilon t})$. If $2 \varepsilon > \sigma_1$, by Lemma \ref{lem:sol-esti} we have that for any $(t, \theta) \in (1, \infty) \times \Sp^{n - 1}$,
\begin{equation}\label{eq:phis1}
\phi = O(e^{- \sigma_1 t}).
\end{equation}
If $2 \varepsilon < \sigma_1$, by Lemma \ref{lem:sol-esti} we get $\phi = O(e^{- 2 \varepsilon t})$. Repeating the same procedure in finite steps, we always have \eqref{eq:phis1}.

{\bf Step 2.} It follows from \eqref{eq:phis1} that $\mathcal{L} \phi = O(e^{- 2 \sigma_1 t}) = O(e^{ - \widetilde{\sigma}_1 t })$. By Lemma \ref{lem:sol-esti}, we have for any $(t, \theta) \in (1, \infty) \times \Sp^{n - 1}$,
$$
\bigg| \phi (t, \theta) - \sum_{i = 1}^{k_1} c_i q_i^+ (t) X_i (\theta) e^{- \sigma_i t} \bigg| \leq C e^{ - \widetilde{\sigma}_1 t },
$$
with some constants $c_i$ for $i = 1, \dots, k_1$.

Denote
\begin{equation}\label{eq:psi1}
\psi_1 (t, \theta) := \sum_{i = 1}^{k_1} c_i q_i^+ (t) X_i (\theta) e^{- \sigma_i t},
\end{equation}
and
$$
\phi_1 (t, \theta) := \phi (t, \theta) - \psi_1 (t, \theta).
$$
Then we have $\mathcal{L} \psi_1 = 0$, and by \eqref{eq:Lphi},
\begin{equation}\label{eq:Lphi1}
\mathcal{L} \phi_1 = \mathcal{F} (\phi).
\end{equation}
For each $\widetilde{\sigma}_i \in \mathcal{S}_2$, we consider nonnegative integers $n_1, \dots, n_k$ such that
\begin{equation}\label{eq:tsigmai}
n_1 \sigma_1 + \cdots + n_k \sigma_k = \widetilde{\sigma}_i ~~~~~~ \textmd{and} ~~~~~~ n_1 + \cdots + n_k \geq 2.
\end{equation}
Notice that there are only finitely many collections of nonnegative integers $n_1, \dots, n_k$ satisfying \eqref{eq:tsigmai}. Set
$$
\widetilde{K}_i = \max \{ n_1 \deg (X_1) + \cdots + n_k \deg (X_k) : n_1, \dots, n_k \in \N ~ \textmd{satisfy \eqref{eq:tsigmai}} \},
$$
and
$$
\widetilde{M}_i = \max \{ m : \deg (X_m) \leq \widetilde{K}_i \}.
$$

{\bf Step 3.} We claim that there exists $\widetilde{\psi}_1$ having the form of
\begin{equation}\label{eq:tpsi1}
\widetilde{\psi}_1 (t, \theta) = \sum_{i = 1}^{l_1} \sum_{j = 0}^{ \widetilde{M}_i } \widetilde{c}_{ij} (t) X_j (\theta) e^{ - \widetilde{\sigma}_i t }
\end{equation}
with $\widetilde{c}_{ij}$ being a smooth periodic function such that for
\begin{equation}\label{eq:tphi1}
\widetilde{\phi}_1 := \phi_1 - \widetilde{\psi}_1,
\end{equation}
we have
\begin{equation}\label{eq:Ltphi1}
\mathcal{L} \widetilde{\phi}_1 = O(e^{ - \widetilde{\sigma}_{l_1 + 1} t }).
\end{equation}

Take some function $\widetilde{\psi}_1$ to be determined later, and define $\widetilde{\phi}_1$ as in \eqref{eq:tphi1}. By \eqref{eq:Lphi1} we get
\begin{equation}\label{eq:Ltphi1-}
\mathcal{L} \widetilde{\phi}_1 = \mathcal{L} \phi_1 - \mathcal{L} \widetilde{\psi}_1 = \mathcal{F} (\phi) - \mathcal{L} \widetilde{\psi}_1.
\end{equation}
We will construct a proper $\widetilde{\psi}_1$ such that
$$
\mathcal{F} (\phi) - \mathcal{L} \widetilde{\psi}_1 = O(e^{ - \widetilde{\sigma}_{l_1 + 1} t }).
$$
Now we expand $\mathcal{F} (\phi)$. Recall that $\phi = O(e^{- \sigma_1 t})$, $\psi_1 = O(e^{- \sigma_1 t})$ and $\phi_1 = O(e^{- 2 \sigma_1 t})$. Therefore,
\begin{equation}\label{eq:Fphi1}
\mathcal{F} (\phi) = \sum_{i = 2}^{[\widetilde{\sigma}_{l_1 + 1}]} a_i \phi^i + O(e^{ - \widetilde{\sigma}_{l_1 + 1} t }) = \sum_{i = 2}^{[\widetilde{\sigma}_{l_1 + 1}]} a_i (\psi_1 + \phi_1)^i + O(e^{ - \widetilde{\sigma}_{l_1 + 1} t }),
\end{equation}
where $a_i = a_i (t)$ is a smooth periodic function (with the same period as $\zeta$). For the lower order terms involving $\phi_1$, we have
\begin{equation}\label{eq:lotphi1}
\psi_1 \phi_1 = O(e^{- 3 \sigma_1 t}) ~~~~~~ \textmd{and} ~~~~~~ \phi_1^2 = O(e^{- 4 \sigma_1 t}).
\end{equation}
For the term involving $\psi_1$ only, by \eqref{eq:psi1} and \cite[Lemma 2.4]{HLL} we have
\begin{equation}\label{eq:sumpsi1}
\aligned
\sum_{i = 2}^{[\widetilde{\sigma}_{l_1 + 1}]} a_i \psi_1^i & = \sum_{n_1 + \cdots + n_k \geq 2}^{[\widetilde{\sigma}_{l_1 + 1}]} a_{n_1 \cdots n_k} (t) X_1^{n_1} \cdots X_k^{n_k} e^{ - (n_1 \sigma_1 + \cdots + n_k \sigma_k) t } + O(e^{ - \widetilde{\sigma}_{l_1 + 1} t }) \\
& = \sum_{i = 1}^{l_1} \sum_{j = 0}^{ \widetilde{M}_i } a_{ij} (t) X_j (\theta) e^{ - \widetilde{\sigma}_i t } + O(e^{ - \widetilde{\sigma}_{l_1 + 1} t }),
\endaligned
\end{equation}
where $a_{ij}$ is a smooth periodic function (with the same period as $\zeta$). Next we discuss it in two cases.

{\it Case 1:} $\widetilde{\sigma}_{l_1} < 3 \sigma_1$. Then we have $\widetilde{\sigma}_{l_1 + 1} \leq 3 \sigma_1$. By \eqref{eq:lotphi1},
$$
\psi_1 \phi_1 = O(e^{ - \widetilde{\sigma}_{l_1 + 1} t }) ~~~~~~ \textmd{and} ~~~~~~ \phi_1^2 = O(e^{ - \widetilde{\sigma}_{l_1 + 1} t }).
$$
This together with \eqref{eq:Ltphi1-}-\eqref{eq:sumpsi1} implies that
$$
\mathcal{L} \widetilde{\phi}_1 = - \mathcal{L} \widetilde{\psi}_1 + \sum_{i = 1}^{l_1} \sum_{j = 0}^{ \widetilde{M}_i } a_{ij} (t) X_j (\theta) e^{ - \widetilde{\sigma}_i t } + O(e^{ - \widetilde{\sigma}_{l_1 + 1} t }).
$$
Since $\mathcal{S}_1 \cap \mathcal{S}_2 = \varnothing$, by Lemma \ref{lem:sol-form} there exists a smooth periodic function $\widetilde{c}_{ij}$ such that
\begin{equation}\label{eq:Ljcijsigmai}
\mathcal{L}_j ( \widetilde{c}_{ij} (t) e^{ - \widetilde{\sigma}_i t } ) = a_{ij} (t) e^{ - \widetilde{\sigma}_i t } ~~~~~~ \textmd{for each} ~ 1 \leq i \leq l_1, 0 \leq j \leq \widetilde{M}_i,
\end{equation}
where $\mathcal{L}_j$ is defined as in \eqref{eq:ckn-Li}. Then we can take $\widetilde{\psi}_1$ to be \eqref{eq:tpsi1} and obtain
$$
\mathcal{L} \widetilde{\phi}_1 = O(e^{ - \widetilde{\sigma}_{l_1 + 1} t }).
$$

{\it Case 2:} $\widetilde{\sigma}_{l_1} \geq 3 \sigma_1$. Let $n_1$ be the largest integer satisfying $\widetilde{\sigma}_{n_1} < 3 \sigma_1$. Then $\widetilde{\sigma}_{n_1 + 1} = 3 \sigma_1 \leq \widetilde{\sigma}_{l_1}$. Now we take $\widetilde{\psi}_{11}$ to be the summation \eqref{eq:tpsi1} from $1$ to $n_1$, i.e,
$$
\widetilde{\psi}_{11} (t, \theta) = \sum_{i = 1}^{n_1} \sum_{j = 0}^{ \widetilde{M}_i } \widetilde{c}_{ij} (t) X_j (\theta) e^{ - \widetilde{\sigma}_i t }.
$$
Then
\begin{equation}\label{eq:Lpsi11}
\mathcal{L} \widetilde{\psi}_{11} = \sum_{i = 1}^{n_1} \sum_{j = 0}^{ \widetilde{M}_i } a_{ij} (t) X_j (\theta) e^{ - \widetilde{\sigma}_i t }.
\end{equation}
Let $\widetilde{\phi}_{11} = \phi_1 - \widetilde{\psi}_{11}$. By \eqref{eq:Lphi1}, \eqref{eq:Fphi1}-\eqref{eq:Lpsi11} we have
$$
\mathcal{L} \widetilde{\phi}_{11} = \mathcal{F} (\phi) - \mathcal{L} \widetilde{\psi}_{11} = O(e^{ - \widetilde{\sigma}_{n_1 + 1} t }).
$$
Since there is no $\sigma_i$ between $\widetilde{\sigma}_1$ and $\widetilde{\sigma}_{n_1 + 1}$, it follows from $\widetilde{\phi}_{11} = O(e^{ - \widetilde{\sigma}_1 t })$ and Lemma \ref{lem:sol-esti} that
$$
\widetilde{\phi}_{11} = O(e^{ - \widetilde{\sigma}_{n_1 + 1} t }).
$$

Now we are in a similar situation as at the beginning of Step 3, with $\widetilde{\sigma}_{n_1 + 1} = 3 \sigma_1$ replacing $\widetilde{\sigma}_1 = 2 \sigma_1$. If $\widetilde{\sigma}_{l_1} < 4 \sigma_1$, we proceed as in Case 1. If $\widetilde{\sigma}_{l_1} \geq 4 \sigma_1$, we proceed as at the beginning of Case 2 by taking the largest integer $n_2$ such that $\widetilde{\sigma}_{n_2} < 4 \sigma_1$. After finitely many steps, we can obtain the desired function $\widetilde{\psi}_1$ in \eqref{eq:tpsi1}, and $\widetilde{\phi}_1$ in \eqref{eq:tphi1}.

{\bf Step 4.} Now we are in a similar situation as in Step 2, with $\widetilde{\sigma}_{l_1 + 1}$ replacing $\widetilde{\sigma}_1$. By \eqref{eq:Ltphi1} and Lemma \ref{lem:sol-esti}, we get
$$
\bigg| \widetilde{\phi}_1 (t, \theta) - \sum_{i = k_1 + 1}^{k_2} c_i q_i^+ (t) X_i (\theta) e^{- \sigma_i t} \bigg| \leq C e^{ - \widetilde{\sigma}_{l_1 + 1} t },
$$
with some constants $c_i$ for $i = k_1 + 1, \dots, k_2$.

Denote
$$
\psi_2 (t, \theta) := \sum_{i = k_1 + 1}^{k_2} c_i q_i^+ (t) X_i (\theta) e^{- \sigma_i t},
$$
and
$$
\phi_2 (t, \theta) := \widetilde{\phi}_1 (t, \theta) - \psi_2 (t, \theta).
$$
Then we have $\mathcal{L} \psi_2 = 0$, and by \eqref{eq:Ltphi1-},
\begin{equation}\label{eq:Lphi2}
\mathcal{L} \phi_2 = \mathcal{F} (\phi) - \mathcal{L} \widetilde{\psi}_1.
\end{equation}

{\bf Step 5.} The argument is similar to Step 3. For some $\widetilde{\psi}_2$ to be determined, set $\widetilde{\phi}_2 = \phi_2 - \widetilde{\psi}_2$. Then by \eqref{eq:Lphi2},
$$
\mathcal{L} \widetilde{\phi}_2 = \mathcal{F} (\phi) - \mathcal{L} \widetilde{\psi}_1 - \mathcal{L} \widetilde{\psi}_2.
$$
We write
$$
\mathcal{F} (\phi) = \sum_{i = 2}^{[\widetilde{\sigma}_{l_2 + 1}]} a_i (\psi_1 + \widetilde{\psi}_1 + \psi_2 + \phi_2)^i + O(e^{- \widetilde{\sigma}_{l_2 + 1} t}).
$$
Recall that, in Step 3 we use $\mathcal{L} \widetilde{\psi}_1$ to cancel the terms with decay $e^{ - \widetilde{\sigma}_i t }$ in the expansion of $\mathcal{F} (\phi)$ for $i = 1, \dots, l_1$. Proceeding similarly, since $\mathcal{S}_1 \cap \mathcal{S}_2 = \varnothing$, we can find $\widetilde{\psi}_2$ having the form of
$$
\widetilde{\psi}_2 (t, \theta) = \sum_{i = l_1 + 1}^{l_2} \sum_{j = 0}^{ \widetilde{M}_i } \widetilde{c}_{ij} (t) X_j (\theta) e^{ - \widetilde{\sigma}_i t }
$$
to cancel the terms with decay $e^{ - \widetilde{\sigma}_i t }$ in the expansion of $\mathcal{F} (\phi)$ for $i = l_1 + 1, \dots, l_2$, where $\widetilde{c}_{ij}$ is a smooth periodic function. Then we conclude that
$$
\mathcal{L} \widetilde{\phi}_2 = O(e^{ - \widetilde{\sigma}_{l_2 + 1} t }).
$$
Repeating the similar arguments as above, we obtain \eqref{eq:ckn-expans} for the case $\mathcal{S}_1 \cap \mathcal{S}_2 = \varnothing$ by denoting the index set $\mathcal{S} [\zeta] = \{ \nu_i \}_{i \geq 1}$.

Next we show \eqref{eq:ckn-expans} the general case, i.e., $\mathcal{S}_1 \cap \mathcal{S}_2 \neq \varnothing$. For an illustration, we consider $\sigma_{k_1} = \widetilde{\sigma}_1$ instead of the strict inequality, which is the smallest element in $\mathcal{S}_1 \cap \mathcal{S}_2$. Let $k_* \in \{ 1, \dots, k_1 - 1 \}$ such that
$$
\sigma_{k_*} < \sigma_{k_* + 1} = \cdots = \sigma_{k_1} = \widetilde{\sigma}_1.
$$
We will modify the discussion above to deal with this case. Now we have
$$
| \phi - \psi_1 | \leq C t e^{ - \widetilde{\sigma}_1 t },
$$
where
\begin{equation}\label{eq:psi1mod}
\psi_1 (t, \theta) = \sum_{i = 1}^{k_*} c_i q_i^+ (t) X_i (\theta) e^{- \sigma_i t}.
\end{equation}
Suppose $\widetilde{\sigma}_{l_1} < 3 \sigma_1$. Let $\psi_1$ be defined as in \eqref{eq:psi1mod} and $\phi_1 = \phi - \psi_1$. Take some function $\widetilde{\psi}_1$ to be determined later, and define $\widetilde{\phi}_1 = \phi_1 - \widetilde{\psi}_1$. By \eqref{eq:Ltphi1-}-\eqref{eq:sumpsi1} we have
$$
\mathcal{L} \widetilde{\phi}_1 = - \mathcal{L} \widetilde{\psi}_1 + \sum_{i = 1}^{l_1} \sum_{j = 0}^{ \widetilde{M}_i } a_{ij} (t) X_j (\theta) e^{ - \widetilde{\sigma}_i t } + O(e^{ - \widetilde{\sigma}_{l_1 + 1} t }).
$$
When $\widetilde{\sigma}_i = \sigma_j \in \mathcal{S}_1 \cap \mathcal{S}_2$ for $1 \leq i \leq l_1$ and $0 \leq j \leq \widetilde{M}_i$, by Lemma \ref{lem:sol-esti}, instead of \eqref{eq:Ljcijsigmai} there exist two smooth periodic functions $\widetilde{b}_{ij}$ and $\widetilde{d}_{ij}$ such that
$$
\mathcal{L}_j ( \widetilde{b}_{ij} (t) e^{ - \widetilde{\sigma}_i t } + \widetilde{d}_{ij} (t) t e^{ - \widetilde{\sigma}_i t } ) = a_{ij} (t) e^{ - \widetilde{\sigma}_i t }.
$$
Hence, we replace those terms $\widetilde{c}_{ij} (t)$ in the summation \eqref{eq:tpsi1} by $\widetilde{b}_{ij} (t) + \widetilde{d}_{ij} (t) t$ to define the new $\widetilde{\psi}_1$ when $\widetilde{\sigma}_i = \sigma_j \in \mathcal{S}_1 \cap \mathcal{S}_2$ for $1 \leq i \leq l_1$ and $0 \leq j \leq \widetilde{M}_i$. Then we get $\mathcal{L} \widetilde{\phi}_1 = O(e^{ - \widetilde{\sigma}_{l_1 + 1} t })$, similar to the result of Step 3.

In the rest of the proof, when $\widetilde{\sigma}_i = \sigma_j \in \mathcal{S}_1 \cap \mathcal{S}_2$, an extra power of $t$ appears when solving $\mathcal{L}_j ( \widetilde{c}_{ij} (t) e^{ - \widetilde{\sigma}_i t } ) = a_{ij} (t) e^{ - \widetilde{\sigma}_i t }$ according to Lemma \ref{lem:sol-form}. Such a power of $t$ will generate more powers of $t$ upon iteration. The worst situation is that $\widetilde{\sigma}_i \in \mathcal{S}_1 \cap \mathcal{S}_2$ for every $i \geq 1$. Therefore, we have at most $m$-th power of $t$ in the front of $e^{ - \nu_{m + 1} t }$ in \eqref{eq:ckn-expans}. The proof of Theorem \ref{thm:ckn-expans} is completed.
\end{proof}

\subsection{Existence}

\begin{proof}[Proof of Theorem \ref{thm:ckn-existw1}] Our target is to find $\phi \in C_\nu^{2, \alpha} ([t_0, \infty) \times \Sp^{n - 1})$ such that
$$
\mathcal{N} (\widehat{w} + \phi) = 0,
$$
where the operator $\mathcal{N}$ is defined as \eqref{eq:ckn-N}. It is equivalent to solving
\begin{equation}\label{42=pmngy7}
\mathcal{L} \phi = - \mathcal{N} (\widehat{w}) + \mathcal{P} (\phi),
\end{equation}
where the operator $\mathcal{P}$ is given by
$$
\mathcal{P} (\phi) = (\widehat{w} + \phi)^{p - 1} - \widehat{w}^{p - 1} - (p - 2) \zeta^{p - 1} \phi.
$$
With the operator $\mathcal{L}^{- 1}$ introduced in Lemma \ref{lem:inverse}, we can rewrite \eqref{42=pmngy7} as
$$
\phi = \mathcal{L}^{- 1} [ - \mathcal{N} (\widehat{w}) + \mathcal{P} (\phi) ].
$$
Define the mapping $\mathcal{T}$ by
$$
\mathcal{T} (\phi) = \mathcal{L}^{- 1} [ - \mathcal{N} (\widehat{w}) + \mathcal{P} (\phi) ].
$$
We claim that $\mathcal{T}$ is a contraction on some closed ball in $C_\nu^{2, \alpha} ([t_0, \infty) \times \Sp^{n - 1})$ for some $t_0$ large. Set
$$
\mathcal{B}_{t_0, R} = \Big\{ \varphi \in C_\nu^{2, \alpha} ([t_0, \infty) \times \Sp^{n - 1}) : \| \varphi \|_{ C_\nu^{2, \alpha} ([t_0, \infty) \times \Sp^{n - 1}) } \leq R \Big\}.
$$

First, we prove that $\mathcal{T}$ maps $\mathcal{B}_{t_0, R}$ to itself for some fixed $R > 0$ and any $t_0$ sufficiently large. By \eqref{eq:hatw-zeta}, there exists $t_1 \geq 1$ such that for any $(t, \theta) \in [t_1, \infty) \times \Sp^{n - 1}$,
$$
\frac{1}{2} \min\nolimits_\R \zeta \leq \widehat{w} (t, \theta) \leq \frac{3}{2} \max\nolimits_\R \zeta.
$$
Let $t_0 \geq t_R + 1$ with
$$
t_R := \max \bigg\{ t_1, \frac{1}{\nu} \ln \bigg( \frac{4 R}{ \min_\R \zeta } \bigg) \bigg\}.
$$
By the definition of $t_R$, one can easily check that $\mathcal{T}$ is well-defined on $\mathcal{B}_{t_0, R}$. Now we estimate $- \mathcal{N} (\widehat{w}) + \mathcal{P} (\phi)$. By \eqref{eq:Nhatw} we have
$$
\| \mathcal{N} (\widehat{w}) \|_{ C_\nu^1 ([t_0, \infty) \times \Sp^{n - 1}) } \leq C_1.
$$
For any $\phi \in \mathcal{B}_{t_0, R}$ with some $R$ to be determined, set
\begin{equation}\label{eq:Qphi-int}
\mathcal{Q} (\phi) := (p - 1) \int_0^1 \Big[ (\widehat{w} + s \phi)^{p - 2} - \zeta^{p - 2} \Big] {\rm d} s.
\end{equation}
Then
$$
\mathcal{P} (\phi) = \phi \mathcal{Q} (\phi).
$$
Note that by \eqref{eq:hatw-zeta},
$$
| (\widehat{w} - \zeta) | + | \nabla (\widehat{w} - \zeta) | \leq \epsilon (t),
$$
where $\epsilon (t)$ is a decreasing function with $\epsilon (t) \to 0$ as $t \to \infty$, and
$$
| \phi | + | \nabla \phi | \leq R e^{- \nu t}.
$$
Then for $t \geq t_0$,
\begin{equation}\label{eq:Qphi-esti}
| \mathcal{Q} (\phi) | + | \nabla \mathcal{Q} (\phi) | \leq C_2 (\epsilon (t) + R e^{- \nu t}).
\end{equation}
Hence
$$
\| \mathcal{P} (\phi) \|_{ C_\nu^1 ([t_0, \infty) \times \Sp^{n - 1}) } \leq C_2 (\epsilon (t_0) + R e^{- \nu t_0}) R.
$$
By Lemma \ref{lem:inverse}, we get
$$
\aligned
\| \mathcal{T} (\phi) \|_{ C_\nu^{2, \alpha} ([t_0, \infty) \times \Sp^{n - 1}) } & \leq C \| - \mathcal{N} (\widehat{w}) + \mathcal{P} (\phi) \|_{ C_\nu^{0, \alpha} ([t_0, \infty) \times \Sp^{n - 1}) } \\
& \leq C [ C_1 + C_2 (\epsilon (t_0) + R e^{- \nu t_0}) R ],
\endaligned
$$
where $C$, $C_1$ and $C_2$ are positive constants independent of $t_0$ and $R$. Take $R > 2 C C_1$ and then take $t_0$ large such that $C C_2 (\epsilon (t_0) + R e^{- \nu t_0}) \leq 1/2$. Then
$$
\| \mathcal{T} (\phi) \|_{ C_\nu^{2, \alpha} ([t_0, \infty) \times \Sp^{n - 1}) } \leq R.
$$
That is, $\mathcal{T} (\mathcal{B}_{t_0, R}) \subset \mathcal{B}_{t_0, R}$.

Next, we verify that $\mathcal{T}$ is a contraction mapping for $t_0$ sufficiently large. For any $\phi_1, \phi_2 \in \mathcal{B}_{t_0, R}$, we have
$$
\mathcal{T} (\phi_1) - \mathcal{T} (\phi_2) = \mathcal{L}^{- 1} [ \mathcal{P} (\phi_1) - \mathcal{P} (\phi_2) ]
$$
and
$$
\aligned
\mathcal{P} (\phi_1) - \mathcal{P} (\phi_2) & = \phi_1 \mathcal{Q} (\phi_1) - \phi_2 \mathcal{Q} (\phi_2) \\
& = (\phi_1 - \phi_2) \mathcal{Q} (\phi_1) + \phi_2 ( \mathcal{Q} (\phi_1) - \mathcal{Q} (\phi_2) ).
\endaligned
$$
By \eqref{eq:Qphi-int}, we also have
$$
\mathcal{Q} (\phi_1) - \mathcal{Q} (\phi_2) = (p - 1) \int_0^1 \Big[ (\widehat{w} + s \phi_1)^{p - 2} - (\widehat{w} + s \phi_2)^{p - 2} \Big] {\rm d} s.
$$
Then
$$
| \mathcal{Q} (\phi_1) - \mathcal{Q} (\phi_2) | + | \nabla ( \mathcal{Q} (\phi_1) - \mathcal{Q} (\phi_2) ) | \leq C ( | \phi_1 - \phi_2 | + | \nabla ( \phi_1 - \phi_2 ) | ).
$$
By \eqref{eq:Qphi-esti}, we obtain that for any $t \geq t_0$,
$$
| \mathcal{P} (\phi_1) - \mathcal{P} (\phi_2) | + | \nabla ( \mathcal{P} (\phi_1) - \mathcal{P} (\phi_2) ) | \leq C (\epsilon (t) + R e^{- \nu t}) ( | \phi_1 - \phi_2 | + | \nabla ( \phi_1 - \phi_2 ) | ),
$$
and hence
$$
\| \mathcal{P} (\phi_1) - \mathcal{P} (\phi_2) \|_{ C_\nu^1 ([t_0, \infty) \times \Sp^{n - 1}) } \leq C (\epsilon (t_0) + R e^{- \nu t_0}) \| \varphi_1 - \varphi_2 \|_{ C_\nu^1 ([t_0, \infty) \times \Sp^{n - 1}) }.
$$
Hence, by Lemma \ref{lem:inverse} we get
$$
\aligned
\| \mathcal{T} (\phi_1) - \mathcal{T} (\phi_2) \|_{ C_\nu^{2, \alpha} ([t_0, \infty) \times \Sp^{n - 1}) } & \leq C (\epsilon (t_0) + R e^{- \nu t_0}) \| \phi_1 - \phi_2 \|_{ C_\nu^1 ([t_0, \infty) \times \Sp^{n - 1}) } \\
& \leq C (\epsilon (t_0) + R e^{- \nu t_0}) \| \phi_1 - \phi_2 \|_{ C_\nu^{2, \alpha} ([t_0, \infty) \times \Sp^{n - 1}) }.
\endaligned
$$
Thus, $\mathcal{T}$ is a contraction when $t_0$ is sufficiently large.

By the contraction mapping principle, there exists $\phi \in C_\nu^{2, \alpha} ([t_0, \infty) \times \Sp^{n - 1})$ satisfying $\mathcal{T} \phi = \phi$. Therefore, $w = \widehat{w} + \phi$ is a solution of \eqref{eq:w} in $(t_0, \infty) \times \Sp^{n - 1}$.
\end{proof}

\appendix

\section{Appendix}\label{Sec=05}

Let $\{ \lambda_i \}_{i \geq 0}$ be the sequence of eigenvalues (in the increasing order with multiplicity) of $- \Delta_\theta$ on $\Sp^{n - 1}$, and $\{ X_i \}_{i \geq 0}$ be a sequence of the corresponding normalized eigenfunctions of $- \Delta_\theta$ in $L^2 (\Sp^{n - 1})$, i.e., for each $i \geq 0$,
$$
- \Delta_\theta X_i = \lambda_i X_i ~~~~~~ \textmd{on} ~ \Sp^{n - 1}.
$$
Recall that $\lambda_0 = 0$, $ \lambda_1 = \cdots = \lambda_n = n - 1$ and $\lambda_{n + 1} = 2 n \dots$, and $X_i$ is a spherical harmonic of certain degree.

\subsection{A lower bound of $\rho_i$}\label{AppendixA}

Let $\{ \rho_i \}_{i \geq 1}$ be the sequence of positive constants defined as in Lemmas \ref{lem:Han21} and \ref{lem:Han22}. Here we give a lower bound of $\rho_i$ for each $i \geq 1$.

\begin{lemma}\label{lem:rhoi2}
Let $\xi$ be a positive solution of \eqref{eq:xi}.
\begin{enumerate}[label = \rm(\roman*)]
\item If $\xi$ is the constant solution, then
$$
\rho_i^2 = \lambda_i - n + 2 ~~~~~~ \textmd{for all} ~ i \geq 1.
$$

\item If $\xi$ is a nonconstant periodic solution, then
$$
\rho_i^2 > \lambda_i - \frac{3 n - 2}{2} ~~~~~~ \textmd{for all} ~ i \geq 1.
$$
\end{enumerate}
\end{lemma}

\begin{proof} If $\xi$ is the positive constant solution of \eqref{eq:xi}, then
$$
K(0) \xi^\frac{4}{n - 2} = \frac{(n - 2)^2}{4} ~~~~~~ \textmd{and} ~~~~~~ L_i = - \frac{ {\rm d}^2 }{ {\rm d} t^2 } + \lambda_i - n + 2.
$$
For every $i \geq 1$, ${\rm Ker} (L_i)$ obviously has a basis $e^{- \rho_i t}$ and $e^{\rho_i t}$, where $\rho_i^2 = \lambda_i - n + 2$.

If $\xi$ is a positive nonconstant periodic solution of \eqref{eq:xi}, then by using the phase plane method we have
\begin{equation}\label{eq:xibound}
0 < K(0) \xi^\frac{4}{n - 2} < \frac{n (n - 2)}{4} ~~~~~~ \textmd{in} ~ \R.
\end{equation}
By Lemma \ref{lem:Han22}, there exists a nonzero smooth periodic function $f$ on $\R$ such that $e^{- \rho_i t} f \in {\rm Ker} (L_i)$. Therefore,
$$
f'' = \bigg[ - \rho_i^2 + \lambda_i + \frac{(n - 2)^2}{4} - \frac{n + 2}{n - 2} K(0) \xi^\frac{4}{n - 2} \bigg] f + 2 \rho_i f'.
$$
Without loss of generality, we may assume that $f(t_0) = \max_{t \in \R} f(t) > 0$. Hence,
$$
0 \geq f'' (t_0) = \bigg[ - \rho_i^2 + \lambda_i + \frac{(n - 2)^2}{4} - \frac{n + 2}{n - 2} K(0) \xi^\frac{4}{n - 2} (t_0) \bigg] f(t_0).
$$
This together with \eqref{eq:xibound} implies that
$$
\rho_i^2 \geq \lambda_i + \frac{(n - 2)^2}{4} - \frac{n + 2}{n - 2} K(0) \xi^\frac{4}{n - 2} (t_0) > \lambda_i - \frac{3 n - 2}{2}.
$$
The proof is completed.
\end{proof}

Since $\lambda_{n + 1} = 2 n$, by Lemma \ref{lem:rhoi2} we immediately get $\rho_{n + 1} > \sqrt{ \frac{n + 2}{2} } \geq 2$ if $n \geq 6$. Recall that the second smallest element in the index set $\mathcal{I} [\xi]$ is $\mu_2 = \min \{ 2, \rho_{n + 1} \}$. Thus we have

\begin{corollary}\label{App-cor}
If $n \geq 6$, then $\mu_2 = 2$.
\end{corollary}

\subsection{Properties of $\mathcal{L}$}\label{AppendixB}

Throughout Appendix \ref{AppendixB}, we always assume that $n \geq 3$, $a$, $b$ and $p$ satisfy \eqref{eq:range}. For a fixed $i \geq 0$ and any $h \in C^2 (\R)$, we write
$$
\mathcal{L} (h X_i) = (\mathcal{L}_i h) X_i,
$$
where $\mathcal{L}$ is defined as in \eqref{eq:ckn-L} and $\{ \lambda_i, X_i (\theta) \}$ is the eigendata of $-\Delta_\theta$ on $\Sp^{n - 1}$. Then
\begin{equation}\label{eq:ckn-Li}
\mathcal{L}_i = - \frac{ {\rm d}^2 }{ {\rm d} t^2 } + \lambda_i + \frac{(n - 2 a - 2)^2}{4} - (p - 1) \zeta^{p - 2}.
\end{equation}

\begin{lemma}\label{lem:ckn-Lc}
Let $\zeta$ be the positive constant solution of \eqref{eq:zeta}.
\begin{enumerate}[label = \rm(\roman*)]
\item For $i = 0$, ${\rm Ker} (\mathcal{L}_0)$ has a basis $\cos (\omega_{a, p} t)$ and $\sin (\omega_{a, p} t)$, where
$$
\omega_{a, p} = \frac{ (n - 2 a - 2) \sqrt{p - 2} }{2}.
$$

\item There exists an increasing sequence of positive constants $\{ \sigma_i \}_{i \geq 1}$, divergent to $\infty$, such that for any $i \geq 1$, ${\rm Ker} (\mathcal{L}_i)$ has a basis $e^{- \sigma_i t}$ and $e^{\sigma_i t}$.
\end{enumerate}
\end{lemma}

\begin{proof} Since $\zeta$ is the positive constant solution of \eqref{eq:zeta}, we have
$$
\zeta^{p - 2} = \frac{(n - 2 a - 2)^2}{4} ~~~~~~ \textmd{and} ~~~~~~ \mathcal{L}_i = - \frac{ {\rm d}^2 }{ {\rm d} t^2 } + \lambda_i - \frac{(p - 2) (n - 2 a - 2)^2}{4}.
$$
Obviously, ${\rm Ker} (\mathcal{L}_0)$ has a basis $\cos (\omega_{a, p} t)$ and $\sin (\omega_{a, p} t)$. When $i \geq 1$,
$$
\lambda_i - \frac{(p - 2) (n - 2 a - 2)^2}{4} \geq n - 1 - \frac{(a + 1 - b) (n - 2 a - 2)^2}{n - 2 + 2 (b - a)} \geq 1,
$$
where we used the facts $0 \leq a < (n - 2)/2$ and $a \leq b < a + 1$ in the last inequality. It follows that for $i \geq 1$, ${\rm Ker} (\mathcal{L}_i)$ has a basis $e^{- \sigma_i t}$ and $e^{\sigma_i t}$, with
$$
\sigma_i = \sqrt{ \lambda_i - \frac{(p - 2) (n - 2 a - 2)^2}{4} }.
$$
The proof is completed.
\end{proof}

When $\zeta$ is a positive nonconstant periodic solution of \eqref{eq:zeta}, we need some ODE theory to analyze the kernel ${\rm Ker} (\mathcal{L}_i)$. When $a = b = 0$, the kernel ${\rm Ker} (\mathcal{L}_i)$ is understood by \cite{KMPS,MP,MPU} very well. Here we will adapt their proofs to study ${\rm Ker} (\mathcal{L}_i)$ for $a$ and $b$ satisfying \eqref{eq:range}. First, we recall some properties of nonconstant periodic solutions of \eqref{eq:zeta}. In \cite{HLW}, Hsia-Lin-Wang have proved that

\begin{lemma}\label{lem:Hamiltonian}
Let $\zeta$ be a positive nonconstant periodic solution of \eqref{eq:zeta}. Then the Hamiltonian energy
$$
H(\zeta', \zeta) = \frac{1}{2} (\zeta')^2 - \frac{(n - 2 a - 2)^2}{8} \zeta^2 + \frac{1}{p} \zeta^p
$$
is a negative constant depending only on $n$, $a$, $b$ and $\min_\R \zeta$.
\end{lemma}

A direct consequence is
\begin{equation}\label{eq:zetaminmax}
0 < \min\nolimits_\R \zeta < \bigg( \frac{(n - 2 a - 2)^2}{4} \bigg)^{1/(p - 2)} < \max\nolimits_\R \zeta < \bigg( \frac{p (n - 2 a - 2)^2}{8} \bigg)^{1/(p - 2)}.
\end{equation}
For $\varepsilon \in (0, \varepsilon_*)$ with $\varepsilon_*^{p - 2} = (n - 2 a - 2)^2/4$, denote $\zeta_\varepsilon$ the positive periodic solution of \eqref{eq:zeta} satisfying $\min\nolimits_\R \zeta_\varepsilon = \zeta_\varepsilon (0) = \varepsilon$. Then any positive nonconstant periodic solution is a translation of $\zeta_\varepsilon$ for some $\varepsilon \in (0, \varepsilon_*)$. Since $\zeta_\varepsilon$ satisfies
$$
\zeta_\varepsilon (0) = \varepsilon ~~~~~~ \textmd{and} ~~~~~~ \zeta_\varepsilon' (0) = 0,
$$
$\zeta_\varepsilon$ is an even function. By Lemma \ref{lem:Hamiltonian}, we know that for all $t \in \R$,
\begin{equation}\label{eq:Hamiltonvalue}
\frac{1}{2} (\zeta_\varepsilon')^2 - \frac{(n - 2 a - 2)^2}{8} \zeta_\varepsilon^2 + \frac{1}{p} \zeta_\varepsilon^p = - \frac{(n - 2 a - 2)^2}{8} \varepsilon^2 + \frac{1}{p} \varepsilon^p.
\end{equation}
In particular,
$$
- \frac{(n - 2 a - 2)^2}{8} (\max\nolimits_\R \zeta_\varepsilon)^2 + \frac{1}{p} (\max\nolimits_\R \zeta_\varepsilon)^p = - \frac{(n - 2 a - 2)^2}{8} \varepsilon^2 + \frac{1}{p} \varepsilon^p.
$$

Let $T_\varepsilon$ be the period of $\zeta_\varepsilon$. Since $\zeta_\varepsilon$ is periodic and even, by \eqref{eq:Hamiltonvalue} there is only one maximum point in $[0, T_\varepsilon]$ and
$$
\zeta_\varepsilon (T_\varepsilon/2) = \max\nolimits_\R \zeta_\varepsilon,
$$
which yields that $\zeta_\varepsilon$ is increasing in $[0, T_\varepsilon/2]$ and decreasing in $[T_\varepsilon/2, T_\varepsilon]$. Hence, by \eqref{eq:Hamiltonvalue},
$$
\aligned
\frac{n - 2 a - 2}{4} T_\varepsilon & = \int_0^{T_\varepsilon/2} \bigg( \zeta_\varepsilon^2 - \frac{8}{p (n - 2 a - 2)^2} \zeta_\varepsilon^p - \varepsilon^2 + \frac{8}{p (n - 2 a - 2)^2} \varepsilon^p \bigg)^{- 1/2} \zeta_\varepsilon' {\rm d} t \\
& = \int_\varepsilon^{\max\nolimits_\R \zeta_\varepsilon} \bigg( s^2 - \frac{8}{p (n - 2 a - 2)^2} s^p - \varepsilon^2 + \frac{8}{p (n - 2 a - 2)^2} \varepsilon^p \bigg)^{- 1/2} {\rm d} s.
\endaligned
$$
It is easy to see that this integral grows as $- \ln \varepsilon + O(1)$.

Now, if $\mathcal{L}_i f = 0$ for a fixed $i \geq 0$, we write
\begin{equation}\label{eq:Z}
Z(t) =
\begin{pmatrix}
f(t) \\
f' (t)
\end{pmatrix}.
\end{equation}
Then
\begin{equation}\label{eq:odes}
Z' (t) =
\begin{pmatrix}
0 & 1 \\
\lambda_i + (n - 2 a - 2)^2/4 - (p - 1) \zeta^{p - 2} & 0
\end{pmatrix}Z(t) =: {\bf A} (t) Z(t),
\end{equation}
which is a linear system with periodic coefficients. By Floquet theorem (see, e,g., \cite[Theorem 5.1]{odebook}), a fundamental matrix of \eqref{eq:odes}, say ${\bf \Phi} (t)$, has the form
$$
{\bf \Phi} (t) = {\bf Q} (t) \exp (t {\bf R}),
$$
where ${\bf Q} (t)$ is a nonsingular periodic matrix and ${\bf R}$ is a constant matrix. Moreover, if $T_\varepsilon$ is the period of $\zeta$, then
$$
{\rm tr}\, {\bf R} = \frac{1}{T_\varepsilon} \int_0^{T_\varepsilon} {\rm tr}\, {\bf A} (t) {\rm d} t = 0
$$
and
$$
\exp (T_\varepsilon {\bf R}) = {\bf \Phi} (0)^{- 1} {\bf \Phi} (T_\varepsilon).
$$
Let $\bf C$ be a nonsingular constant matrix. Then ${\bf \Psi} (t) = {\bf \Phi} (t) {\bf C}$ is another fundamental matrix of \eqref{eq:odes}. It is easy to see that
$$
{\bf \Psi} (t) = ({\bf Q} (t) {\bf C}) \exp (t {\bf C}^{- 1} {\bf R} {\bf C}).
$$
Here ${\bf C}^{- 1} {\bf R} {\bf C}$ is similar to ${\bf R}$ and ${\bf Q} (t) {\bf C}$ is also a nonsingular periodic matrix. Since $\bf R$ has trace zero, we can take a proper matrix $\bf C$ such that ${\bf C}^{- 1} {\bf R} {\bf C}$ is equal to one of the following:
$$
\begin{pmatrix}
0 & 0 \\
0 & 0
\end{pmatrix}, ~~~
\begin{pmatrix}
0 & 1 \\
0 & 0
\end{pmatrix}, ~~~
\begin{pmatrix}
- \sigma & 0 \\
0 & \sigma
\end{pmatrix}, ~~~
\begin{pmatrix}
0 & \sigma \\
- \sigma & 0
\end{pmatrix} ~~~ \textmd{or} ~~~
\begin{pmatrix}
- \sigma - \kappa \sqrt{- 1} & 0 \\
0 & \sigma + \kappa \sqrt{- 1}
\end{pmatrix},
$$
for some constants $\sigma > 0$ and $\kappa \in \R \setminus \{ 0 \}$. Therefore, the corresponding exponent matrix $\exp (t {\bf C}^{- 1} {\bf R} {\bf C})$ is equal to one of the following:
$$
\begin{pmatrix}
1 & 0 \\
0 & 1
\end{pmatrix}, ~~~
\begin{pmatrix}
1 & t \\
0 & 1
\end{pmatrix}, ~~~
\begin{pmatrix}
e^{- \sigma t} & 0 \\
0 & e^{\sigma t}
\end{pmatrix}, ~~~
\begin{pmatrix}
\cos (\sigma t) & \sin (\sigma t) \\
- \sin (\sigma t) & \cos (\sigma t)
\end{pmatrix},
$$
or
$$
\begin{pmatrix}
e^{- \sigma t} (\cos (\kappa t) - \sin (\kappa t) \sqrt{- 1}) & 0 \\
0 & e^{\sigma t} (\cos (\kappa t) + \sin (\kappa t) \sqrt{- 1})
\end{pmatrix}.
$$
Let $( q_i^+ + p_i^+ \sqrt{- 1}, q_i^- + p_i^- \sqrt{- 1} )$ be the first row of ${\bf Q} (t) {\bf C}$, where $q_i^+$, $p_i^+$, $q_i^-$ and $p_i^-$ are real periodic functions, with the same period as $\zeta$. By \eqref{eq:Z}, the first row of $\bf \Psi$ is a basis of ${\rm Ker} (\mathcal{L}_i)$. Thus, we obtain that one of the following types is a basis of ${\rm Ker} (\mathcal{L}_i)$:
\begin{itemize}
\item Type I: $q_i^+$ and $q_i^-$,

\item Type II: $q_i^+$ and $q_i^- + t q_i^+$,

\item Type III: $q_i^+ e^{- \sigma_i t}$ and $q_i^- e^{\sigma_i t}$, with $\sigma_i > 0$,

\item Type IV: $q_i^+ \cos (\sigma_i t) - q_i^- \sin (\sigma_i t)$ and $q_i^+ \sin (\sigma_i t) + q_i^- \cos (\sigma_i t)$, with $\sigma_i > 0$,

\item Type V: $e^{- \sigma_i t} ( q_i^+ \cos (\kappa_i t) + p_i^+ \sin (\kappa_i t) )$ and $e^{\sigma_i t} ( q_i^- \cos (\kappa_i t) - p_i^- \sin (\kappa_i t) )$, with $\sigma_i > 0$ and $\kappa_i \in \R \setminus \{ 0 \}$.
\end{itemize}

\begin{lemma}\label{lem:ckn-L0}
Let $\zeta$ be a positive nonconstant periodic solution of \eqref{eq:zeta}. Then ${\rm Ker} (\mathcal{L}_0)$ has a basis $q_0^+$ and $q_0^- + c t q_0^+$, for some smooth periodic functions $q_0^+$ and $q_0^-$ on $\R$, with the same period as $\zeta$, and some constant $c$.
\end{lemma}

\begin{proof} Let $\varepsilon = \min\nolimits_\R \zeta$ and let $\zeta_\varepsilon$ be the positive periodic solution of \eqref{eq:zeta} with $\varepsilon = \zeta_\varepsilon (0) = \min\nolimits_\R \zeta_\varepsilon$. Then $\zeta (t) = \zeta_\varepsilon (t + T)$ for some $T \in [0, T_\varepsilon)$. Consider the following two families of solutions of \eqref{eq:zeta}:
$$
\{ \zeta_\varepsilon (t + T + \tau) \}_{\tau \in \R} ~~~~~~ \textmd{and} ~~~~~~ \{ \zeta_{\varepsilon + \delta} (t + T) \}_{\delta \in (- \varepsilon, \varepsilon_* - \varepsilon)}.
$$
Then
$$
\Phi_0^+ (t) := \frac{ {\rm d} }{ {\rm d} \tau } \bigg|_{\tau = 0} \zeta_\varepsilon (t + T + \tau) ~~~~~~ \textmd{and} ~~~~~~ \Phi_0^- (t) := \frac{ {\rm d} }{ {\rm d} \delta } \bigg|_{\delta = 0} \zeta_{\varepsilon + \delta} (t + T)
$$
are two solutions in ${\rm Ker} (\mathcal{L}_0)$. Obviously, $\Phi_0^+ (t) = \zeta' (t)$, $\Phi_0^+ (- T) = 0$ and $\Phi_0^- (- T) = 1$. So $\Phi_0^+$ and $\Phi_0^-$ are two linearly independent functions. Moreover, since
$$
\zeta_{\varepsilon + \delta} (t + T) = \zeta_{\varepsilon + \delta} (t + T + T_{\varepsilon + \delta}),
$$
differentiating with respect to $\delta$ at $0$ we get
\begin{equation}\label{eq:Phi0-}
\Phi_0^- (t) = \Phi_0^- (t + T_\varepsilon) + \Phi_0^+ (t) T_\varepsilon'.
\end{equation}
By induction, we have
\begin{equation}\label{eq:Phi0-k}
\Phi_0^- (t + k T_\varepsilon) = \Phi_0^- (t) - k \Phi_0^+ (t) T_\varepsilon' ~~~~~~ \textmd{for} ~ k \in \Z.
\end{equation}
If $T_\varepsilon' = 0$, then by \eqref{eq:Phi0-}, $\Phi_0^-$ is a periodic function with the same period as $\zeta$, and thus $\{ \Phi_0^+, \Phi_0^- \}$ is a basis of Type I. If $T_\varepsilon' \neq 0$, then by \eqref{eq:Phi0-k}, $\{ \Phi_0^+, \Phi_0^- \}$ is a basis of Type II.
\end{proof}

\begin{lemma}\label{lem:ckn-Ln}
Let $\zeta$ be a positive nonconstant periodic solution of \eqref{eq:zeta}. Then there exists an increasing sequence of positive constants $\{ \sigma_i \}_{i \geq 1}$, divergent to $\infty$, such that for any $i \geq 1$, ${\rm Ker} (\mathcal{L}_i)$ has a basis $e^{- \sigma_i t} q_i^+$ and $e^{\sigma_i t} q_i^-$, for some smooth periodic functions $q_i^+$ and $q_i^-$ on $\R$, with the same period as $\zeta$.
\end{lemma}

\begin{proof} We first study the operator $\mathcal{L}_i$ for $i \geq n + 1$. When $i \geq n + 1$, we have
$$
\aligned
\lambda_i + \frac{(n - 2 a - 2)^2}{4} - (p - 1) \zeta^{p - 2} & > 2 n + \frac{(n - 2 a - 2)^2}{4} - \frac{p (p - 1) (n - 2 a - 2)^2}{8} \\
& \geq 2 n - \frac{(p - 2) (p + 1) (n - 2 a - 2)^2}{8} \\
& \geq \frac{n + 2}{2},
\endaligned
$$
where we used \eqref{eq:zetaminmax} in the first inequality and $0 \leq a < (n - 2)/2$, $a \leq b < a + 1$ in the last inequality. Therefore, for $i \geq n + 1$, $\mathcal{L}_i$ satisfies the maximum principle, i.e., if $\mathcal{L}_i f = 0$, then $f$ does not admit local positive maximum and local negative minimum. Obviously, functions of Type I and Type II can not be a basis of ${\rm Ker} (\mathcal{L}_i)$.

Suppose that $f = q_i^+ \cos (\sigma_i t) - q_i^- \sin (\sigma_i t) \in {\rm Ker} (\mathcal{L}_i)$. Without loss of generality, we can assume that $q_i^+ (0) = q_i^- (0) = 1$. Then we have for $k \in \Z$,
\begin{equation}\label{eq:fkTepsilon}
f(k T_\varepsilon) = \cos (k \sigma_i T_\varepsilon) - \sin (k \sigma_i T_\varepsilon) = \sqrt{2} \cos \bigg( k \sigma_i T_\varepsilon + \frac{\pi}{4} \bigg).
\end{equation}
If $\sigma_i T_\varepsilon/\pi \in \Q$, then $f$ is a periodic function, which can not belong to ${\rm Ker} (\mathcal{L}_i)$. If $\sigma_i T_\varepsilon/\pi \notin \Q$, it follows from \eqref{eq:fkTepsilon} and Weyl's equidistribution theorem that $f$ admits local positive maximum and local negative minimum. This is a contradiction.

Suppose that $h = e^{- \sigma_i t} ( q_i^+ \cos (\kappa_i t) + p_i^+ \sin (\kappa_i t) ) \in {\rm Ker} (\mathcal{L}_i)$. Without loss of generality, we can assume that $q_i^+ (0) = p_i^+ (0) = 1$. Then we have for $k \in \Z$,
\begin{equation}\label{eq:hkTepsilon}
h(k T_\varepsilon) = e^{- k \sigma_i T_\varepsilon} ( \cos (k \kappa_i T_\varepsilon) + \sin (k \kappa_i T_\varepsilon) ) = \sqrt{2} e^{- k \sigma_i T_\varepsilon} \cos \bigg( k \kappa_i T_\varepsilon - \frac{\pi}{4} \bigg).
\end{equation}
Set
$$
\Psi_i^+ (t) := \exp \bigg( - t \sqrt{ \lambda_i + \frac{(n - 2 a - 2)^2}{4} } \bigg).
$$
Then
$$
\mathcal{L}_i (h - \Psi_i^+) = (p - 1) \zeta^{p - 2} \Psi_i^+ > 0.
$$
Note that $(h - \Psi_i^+) (0) = 0$ and $\lim\nolimits_{t \to + \infty} (h - \Psi_i^+) (t) = 0$. By the maximum principle of $\mathcal{L}_i$, we have
$$
h(t) = e^{- \sigma_i t} ( q_i^+ \cos (\kappa_i t) + p_i^+ \sin (\kappa_i t) ) \geq \exp \bigg( - t \sqrt{ \lambda_i + \frac{(n - 2 a - 2)^2}{4} } \bigg) ~~~~~~ \textmd{for} ~ t \geq 0.
$$
This together with \eqref{eq:hkTepsilon} implies that
$$
\cos \bigg( k \kappa_i T_\varepsilon - \frac{\pi}{4} \bigg) > 0 ~~~~~~ \textmd{for all} ~ k \in \N.
$$
Hence, $\kappa_i T_\varepsilon/(2 \pi) \in \Z$. It follows that $T_\varepsilon$ is a period of $\cos (\kappa_i t)$, and thus $h$ is in Type III.

Therefore, for $i \geq n + 1$, a basis of ${\rm Ker} (\mathcal{L}_i)$ must be of Type III. Let $q_i^+ e^{- \sigma_i t}$ and $q_i^- e^{\sigma_i t}$ be a basis of ${\rm Ker} (\mathcal{L}_i)$. Without loss of generality, we can assume that $q_i^+ (0) = q_i^- (0) = 1$. As above we have
$$
\mathcal{L}_i (q_i^+ e^{- \sigma_i t} - \Psi_i^+) = (p - 1) \zeta^{p - 2} \Psi_i^+ > 0.
$$
Note $(q_i^+ e^{- \sigma_i t} - \Psi_i^+) (0) = 0$ and $\lim\nolimits_{t \to + \infty} (q_i^+ e^{- \sigma_i t} - \Psi_i^+) (t) = 0$. By the maximum principle of $\mathcal{L}_i$, we have
$$
q_i^+ (t) e^{- \sigma_i t} \geq \exp \bigg( - t \sqrt{ \lambda_i + \frac{(n - 2 a - 2)^2}{4} } \bigg) ~~~~~~ \textmd{for} ~ t \geq 0.
$$
Thus, $q_i^+$ is a positive periodic function on $\R$.

Now we prove $\sigma_{i + 1} \geq \sigma_i$ for $i \geq n + 1$. Notice that
$$
\mathcal{L}_{i + 1} (q_i^+ e^{- \sigma_i t} - q_{i + 1}^+ e^{ - \sigma_{i + 1} t }) = (\lambda_{i + 1} - \lambda_i) q_i^+ e^{- \sigma_i t} \geq 0,
$$
$(q_i^+ e^{- \sigma_i t} - q_{i + 1}^+ e^{ - \sigma_{i + 1} t }) (0) = 0$ and $\lim\nolimits_{t \to + \infty} (q_i^+ e^{- \sigma_i t} - q_{i + 1}^+ e^{ - \sigma_{i + 1} t }) (t) = 0$. By the maximum principle of $\mathcal{L}_{i + 1}$, we have
$$
e^{ (\sigma_{i + 1} - \sigma_i) t } \geq \frac{ q_{i + 1}^+ (t) }{q_i^+ (t)} \geq \frac{ \min\nolimits_\R q_{i + 1}^+ }{ \max\nolimits_\R q_i^+ } > 0 ~~~~~~ \textmd{for} ~ t \geq 0.
$$
Thus, $\sigma_{i + 1} \geq \sigma_i$ for $i \geq n + 1$.

Finally, we study the space ${\rm Ker} (\mathcal{L}_1) = \cdots = {\rm Ker} (\mathcal{L}_n)$. Inspired by the argument in \cite[page 242]{KMPS}, we consider the conjugate operator $\zeta^{- p/2} \mathcal{L}_1 \zeta^{p/2}$. Then
$$
\zeta^{- p/2} \mathcal{L}_1 \zeta^{p/2} = - \frac{ {\rm d}^2 }{ {\rm d} t^2 } - p \zeta^{- 1} \zeta' \frac{ {\rm d} }{ {\rm d} t} + \bigg[ n - 1 - \frac{(p - 2) (p + 2) (n - 2 a - 2)^2}{16} - \frac{p (p - 2)}{2} \zeta^{- 2} H \bigg],
$$
where $H \equiv H(\zeta', \zeta) < 0$ is the Hamiltonian energy of $\zeta$. Moreover,
$$
n - 1 - \frac{(p - 2) (p + 2) (n - 2 a - 2)^2}{16} \geq n - 1 - \frac{4}{n - 2} \cdot \frac{4 (n - 1)}{n - 2} \cdot \frac{(n - 2)^2}{16} = 0,
$$
where we used the facts $0 \leq a < (n - 2)/2$ and $a \leq b < a + 1$ in the last inequality. Therefore, $\zeta^{- p/2} \mathcal{L}_1 \zeta^{p/2}$ satisfies the maximum principle, i.e., if $\zeta^{- p/2} \mathcal{L}_1 (\zeta^{p/2} h) = 0$, then $h$ does not admit local positive maximum and local negative minimum. Furthermore,
$$
{\rm Ker} (\zeta^{- p/2} \mathcal{L}_1 \zeta^{p/2}) = \zeta^{- p/2} {\rm Ker} (\mathcal{L}_1).
$$
It implies that a basis of ${\rm Ker} (\zeta^{- p/2} \mathcal{L}_1 \zeta^{p/2})$ must be one of Type I-V. As above, since $\zeta^{- p/2} \mathcal{L}_1 \zeta^{p/2}$ satisfies the maximum principle, a basis of ${\rm Ker} (\zeta^{- p/2} \mathcal{L}_1 \zeta^{p/2})$ must be of Type III, so is a basis of ${\rm Ker} (\mathcal{L}_1)$. Following the same idea, one can easily prove that $\sigma_{n + 1} > \sigma_n = \cdots = \sigma_1$. Eventually, as in the proof of Lemma \ref{lem:rhoi2}, we have
$$
\sigma_i^2 \geq \lambda_i - \frac{(p - 2) (p + 1) (n - 2 a - 2)^2}{8}.
$$
Thus, $\lim_{i \to \infty} \sigma_i = + \infty$. The proof of Lemma \ref{lem:ckn-Ln} is completed.
\end{proof}

Once Lemmas \ref{lem:ckn-Lc}, \ref{lem:ckn-L0} and \ref{lem:ckn-Ln} hold, adapting the proofs in \cite[Lemmas A.2 and A.8]{HLL} and in \cite[Theorem 2.9]{HL}, we have

\begin{lemma}\label{lem:sol-form}
Let $\gamma > 0$ be a constant, $m \geq 0$ be an integer, and $h$ be a smooth periodic function with the same period as $\zeta$. Then for $i \geq 0$, the equation
$$
\mathcal{L}_i \phi = t^m e^{- \gamma t} h(t)
$$
has a particular solution $\phi_*$ having the form of, for $\gamma \neq \sigma_i$,
$$
\phi_* (t) = e^{- \gamma t} \sum_{j = 0}^m t^j r_j (t),
$$
and for $\gamma = \sigma_i$,
$$
\phi_* (t) = e^{- \gamma t} \sum_{j = 0}^{m + 1} t^j r_j (t),
$$
where $r_0, \dots, r_{m + 1}$ are smooth periodic functions, with the same period as $\zeta$.
\end{lemma}

\begin{lemma}\label{lem:sol-esti}
Let $\gamma > 0$ be a constant, $m \geq 0$ be an integer, and $f$ be a continuous function in $[1, \infty) \times \Sp^{n - 1}$, for all $(t, \theta) \in [1, \infty) \times \Sp^{n - 1}$,
$$
| f(t, \theta) | \leq C t^m e^{- \gamma t}.
$$
Let $\phi$ be a solution of $\mathcal{L} \phi = f$ in $(1, \infty) \times \Sp^{n - 1}$ such that $\phi (t, \theta) \to 0$ as $t \to \infty$ uniformly for $\theta \in \Sp^{n - 1}$.
\begin{enumerate}[label = \rm(\roman*)]
\item If $0 < \gamma \leq \sigma_1$, then for any $(t, \theta) \in (1, \infty) \times \Sp^{n - 1}$,
$$
| \phi (t, \theta) | \leq \bigg\{
\aligned
& C t^m e^{- \gamma t} ~~~~~~ & \textmd{if} & ~ 0 < \gamma < \sigma_1, \\
& C t^{m + 1} e^{- \gamma t} ~~~~~~ & \textmd{if} & ~ \gamma = \sigma_1.
\endaligned
$$

\item If $\sigma_l < \gamma \leq \sigma_{l + 1}$ for some positive integer $l$, then for any $(t, \theta) \in (1, \infty) \times \Sp^{n - 1}$,
$$
\bigg| \phi (t, \theta) - \sum_{i = 1}^l c_i q_i^+ (t) X_i (\theta) e^{- \sigma_i t} \bigg| \leq
\bigg\{
\aligned
& C t^m e^{- \gamma t} ~~~~~~ & \textmd{if} & ~ \sigma_l < \gamma < \sigma_{l + 1}, \\
& C t^{m + 1} e^{- \gamma t} ~~~~~~ & \textmd{if} & ~ \gamma = \sigma_{l + 1},
\endaligned
$$
with some constants $c_i$ for $i = 1, \dots, l$.
\end{enumerate}
\end{lemma}

\begin{lemma}\label{lem:inverse}
Let $\alpha \in (0, 1)$, $\nu > \sigma_1$ with $\nu \neq \sigma_i$ for any $i$, and $f \in C_\nu^{0, \alpha} ([t_0, \infty) \times \Sp^{n - 1})$. Then the equation
$$
\mathcal{L} \phi = f ~~~~~~ \textmd{in} ~ (t_0, \infty) \times \Sp^{n - 1}
$$
admits a solution $\phi \in C_\nu^{2, \alpha} ([t_0, \infty) \times \Sp^{n - 1})$, and
$$
\| \phi \|_{ C_\nu^{2, \alpha} ([t_0, \infty) \times \Sp^{n - 1}) } \leq C \| f \|_{ C_\nu^{0, \alpha} ([t_0, \infty) \times \Sp^{n - 1}) },
$$
where $C$ is a positive constant depending only on $n$, $a$, $b$, $\alpha$, $\nu$ and $\zeta$, independent of $t_0$. Moreover, the correspondence $f \mapsto \phi$ is linear.

Denote the correspondence $f \mapsto \phi$ by $\mathcal{L}^{- 1}$. Then
$$
\mathcal{L}^{- 1} : C_\nu^{0, \alpha} ([t_0, \infty) \times \Sp^{n - 1}) \to C_\nu^{2, \alpha} ([t_0, \infty) \times \Sp^{n - 1})
$$
is a bounded linear operator, and the bound of $\mathcal{L}^{- 1}$ does not depend on $t_0$.
\end{lemma}

\bigskip
\bigskip

\noindent Xusheng Du

\noindent School of Mathematical Sciences, Laboratory of Mathematics and Complex Systems, MOE \\
Beijing Normal University, Beijing 100875, China \\[1mm]
Email: \textsf{xduah@bnu.edu.cn}

\bigskip

\noindent Hui Yang

\noindent School of Mathematical Sciences \\
Shanghai Jiao Tong University, Shanghai 200240, China \\[1mm]
Email: \textsf{hui-yang@sjtu.edu.cn}


\begin{thebibliography}{10}\addcontentsline{toc}{section}{References}

\bibitem{CGS} L. A. Caffarelli, B. Gidas and J. Spruck, Asymptotic symmetry and local behavior of semilinear elliptic equations with critical Sobolev growth. {\it Comm. Pure Appl. Math.} {\bf 42} (1989), no. 3, 271-297.

\bibitem{CKN} L. A. Caffarelli, R. Kohn and L. Nirenberg, First order interpolation inequalities with weights. {\it Compositio Math.} {\bf 53} (1984), no. 3, 259-275.

\bibitem{CW} F. Catrina and Z.-Q. Wang, On the Caffarelli-Kohn-Nirenberg inequalities: sharp constants, existence (and nonexistence), and symmetry of extremal functions. {\it Comm. Pure Appl. Math.} {\bf 54} (2001), no. 2, 229-258.

\bibitem{CD} H. Chan and A. DelaTorre, An analytic construction of singular solutions related to a critical Yamabe problem. {\it Comm. Partial Differential Equations} {\bf 45} (2020), no. 11, 1621-1646.

\bibitem{CHanY} S.-Y. A. Chang, Z.-C. Han and P. C. Yang, Some remarks on the geometry of a class of locally conformally flat metrics. Geometric analysis, 37-56, Progr. Math., 333, Birkh\"auser/Springer, Cham, [2020].

\bibitem{CL97} C.-C. Chen and C.-S. Lin, Estimates of the conformal scalar curvature equation via the method of moving planes. {\it Comm. Pure Appl. Math.} {\bf 50} (1997), no. 10, 971-1017.

\bibitem{CL98} C.-C. Chen, C.-S. Lin, Estimates of the conformal scalar curvature equation via the method of moving planes, II. {\it J. Differential Geom.} {\bf 49} (1998), no. 1, 115-178.

\bibitem{CL99a} C.-C. Chen and C.-S. Lin, On the asymptotic symmetry of singular solutions of the scalar curvature equations. {\it Math. Ann.} {\bf 313} (1999), no. 2, 229-245.

\bibitem{CL99} C.-C. Chen and C.-S. Lin, Existence of positive weak solutions with a prescribed singular set of semilinear elliptic equations. {\it J. Geom. Anal.} {\bf 9} (1999), no. 2, 221-246.

\bibitem{CC} K.S. Chou and C.W. Chu, On the best constant for a weighted Sobolev-Hardy inequality. {\it J. London Math. Soc. (2)} {\bf 48} (1993), no. 1, 137-151.

\bibitem{odebook} E. A. Coddington and N. Levinson, Theory of ordinary differential equations, McGraw-Hill Book Co., Inc., New York-Toronto-London, 1955, xii+429 pp.

\bibitem{DEL} J. Dolbeault, M. J. Esteban and M. Loss, Rigidity versus symmetry breaking via nonlinear flows on cylinders and Euclidean spaces. {\it Invent. Math.} {\bf 206} (2016), no. 2, 397-440.

\bibitem{FS} V. Felli and M. Schneider, Perturbation results of critical elliptic equations of Caffarelli-Kohn-Nirenberg type. {\it J. Differential Equations} {\bf 191} (2003), no. 1, 121-142.

\bibitem{GLW} Z. Guo, J. Li and F. Wan, Asymptotic behavior at the isolated singularities of solutions of some equations on singular manifolds with conical metrics. {\it Comm. Partial Differential Equations} {\bf 45} (2020), no.12, 1647-1681.

\bibitem{HLL} Q. Han, X. Li and Y. Li, Asymptotic expansions of solutions of the Yamabe equation and the $\sigma_k$-Yamabe equation near isolated singular points. {\it Comm. Pure Appl. Math.} {\bf 74} (2021), no. 9, 1915-1970.

\bibitem{HL} Q. Han and Y. Li, Singular solutions to the Yamabe equation with prescribed asymptotics. {\it J. Differential Equations} {\bf 274} (2021), 127-150.

\bibitem{HLT} Z.-C. Han, Y.Y. Li and E. V. Teixeira. Asymptotic behavior of solutions to the $\sigma_k$-Yamabe equation near isolated singularities. {\it Invent. Math.} {\bf 182} (2010), no. 3, 635-684.

\bibitem{HXZ} Z.-C. Han, J. Xiong, L. Zhang, Asymptotic behavior of solutions to the Yamabe equation with an asymptotically flat metric. {\it J. Funct. Anal.} {\bf 285} (2023), no. 4, Paper No. 109982, 55 pp.

\bibitem{HLW} C.-H. Hsia, C.-S. Lin and Z.-Q. Wang, Asymptotic symmetry and local behavior of solutions to a class of anisotropic elliptic equations. {\it Indiana Univ. Math. J.} {\bf 60} (2011), no. 5, 1623-1654.

\bibitem{KMPS} N. Korevaar, R. Mazzeo, F. Pacard and R. Schoen, Refined asymptotics for constant scalar curvature metrics with isolated singularities. {\it Invent. Math.} {\bf 135} (1999), no. 2, 233-272.

\bibitem{LC96} C. Li, Local asymptotic symmetry of singular solutions to nonlinear elliptic equations. {\it Invent. Math.} {\bf123} (1996), no. 2, 221-231.

\bibitem{L00} C.-S. Lin, Estimates of the scalar curvature equation via the method of moving planes III. {\it Comm. Pure Appl. Math.} {\bf 53} (2000), no. 5, 611-646.

\bibitem{LW} C.-S. Lin and Z.-Q. Wang, Symmetry of extremal functions for the Caffarrelli-Kohn-Nirenberg inequalities. {\it Proc. Amer. Math. Soc.} {\bf 132} (2004), no. 6, 1685-1691.

\bibitem{Mar} F. Marques, Isolated singularities of solutions to the Yamabe equation. {\it Calc. Var. Partial Differential Equations} {\bf 32} (2008), no. 3, 349-371.

\bibitem{MP96} R. Mazzeo and F. Pacard, A construction of singular solutions for a semilinear elliptic equation using asymptotic analysis. {\it J. Differential Geom.} {\bf 44} (1996), no. 2, 331-370.

\bibitem{MP} R. Mazzeo and F. Pacard, Constant scalar curvature metrics with isolated singularities. {\it Duke Math. J.} {\bf 99} (1999), no. 3, 353-418.

\bibitem{MPU} R. Mazzeo, D. Pollack and K. Uhlenbeck, Moduli spaces of singular Yamabe metrics. {\it J. Amer. Math. Soc.} {\bf 9} (1996), no. 2, 303-344.

\bibitem{P} F. Pacard, Solutions with high dimensional singular set, to a conformally invariant elliptic equation in $\R^4$ and in $\R^6$. {\it Comm. Math. Phys.} {\bf 159} (1994), no. 2, 423-432.

\bibitem{Sch88} R. Schoen, The existence of weak solutions with prescribed singular behavior for a conformally invariant scalar equation. {\it Comm. Pure Appl. Math.} {\bf 41} (1988), no. 3, 317-392.

\bibitem{SY88} R. Schoen and S.-T. Yau, Conformally flat manifolds, Kleinian groups and scalar curvature. {\it Invent. Math.} {\bf 92} (1988), no. 1, 47-71.

\bibitem{TZ06} S. Taliaferro and L. Zhang, Asymptotic symmetries for conformal scalar curvature equations with singularity. {\it Calc. Var. Partial Differential Equations} {\bf 26} (2006), no. 4, 401-428.

\bibitem{XZ} J. Xiong and L. Zhang, Isolated singularities of solutions to the Yamabe equation in dimension $6$. {\it Int. Math. Res. Not. IMRN} (2022), no. 12, 9571-9597.

\bibitem{Zhang02} L. Zhang, Refined asymptotic estimates for conformal scalar curvature equation via moving sphere method. {\it J. Funct. Anal.} {\bf 192} (2002), no. 2, 491-516.

\end{thebibliography}
\end{document}